\documentclass[12pt, 
]{amsart}

\usepackage{amssymb}
\usepackage[cmtip,all]{xy}

\newtheorem{theo}{Theorem}[section]
\newtheorem{prop}[theo]{Proposition}
\newtheorem{lemm}[theo]{Lemma}
\newtheorem{cor}[theo]{Corollary}
\newtheorem{claim}[theo]{Claim}

\theoremstyle{definition}
\newtheorem{defi}[theo]{Definition}

\newtheorem{step}{Step}

\theoremstyle{remark}
\newtheorem{rem}[theo]{Remark}

\numberwithin{equation}{section}

\newcommand{\deldel}{\sqrt{-1}\partial \overline{\partial}}

\newcommand{\dbar}{\overline{\partial}}
\newcommand{\e}{\varepsilon}
\newcommand{\ome}{\widetilde{\omega}}
\newcommand{\I}[1]{\mathcal{I}(#1)}

\newcommand{\lla}[0]{{\langle\!\hspace{0.02cm} \!\langle}}
\newcommand{\rra}[0]{{\rangle\!\hspace{0.02cm}\!\rangle}}

\begin{document}

\title[]{An injectivity theorem \\
with multiplier ideal sheaves of \\
singular metrics with transcendental singularities}
 
\author{Shin-ichi Matsumura}

\address{Mathematical Institute, Tohoku University, 
6-3, Aramaki Aza-Aoba, Aoba-ku, Sendai 980-8578, Japan.}

\email{mshinichi@m.tohoku.ac.jp, mshinichi0@gmail.com}

\thanks{}



\maketitle

\begin{abstract}
The purpose of this paper is to establish 
an injectivity theorem generalized to    
pseudo-effective line bundles with 
transcendental (non-algebraic) singular hermitian metrics 
and multiplier ideal sheaves. 
As an application, we obtain a Nadel type vanishing theorem. 
For the proof, we study the asymptotic behavior of the harmonic forms 
with respect to a family of regularized metrics,   
and give a method to obtain  $L^{2}$-estimates 
of solutions of the $\dbar$-equation 
by using the de Rham-Weil isomorphism between the $\dbar$-cohomology and  
the $\rm{\check{C}}$ech cohomology.
\end{abstract}

\section{Introduction}\label{S1}
The Kodaira vanishing theorem and its generalizations 
play an important role when we consider certain fundamental problems 
in higher dimensional algebraic geometry 
(in particular birational geometry). 
The following result proved by Koll$\rm{\acute{a}}$r, 
the so-called injectivity theorem, is a celebrated generalization 
of the Kodaira vanishing theorem in algebraic geometry. 
In this paper, we study the injectivity theorem 
from the viewpoint of complex differential geometry and the theory of 
several complex variables.
Our purpose is to establish an analytic version of 
the injectivity theorem 
formulated for pseudo-effective line bundles equipped with transcendental 
singular (hermitian) metrics by using multiplier ideal sheaves.

\begin{theo}[{\cite{Kol86}, cf.\,\cite{EV92}}] \label{Kol}
Let $F$ be a semi-ample line bundle on a smooth 
projective variety $X$. 
Then, for a $($non-zero$)$ section $s$ of a positive multiple 
$F^{m}$ of the line bundle $F$, 
the multiplication map 
induced by the tensor product with $s$ 
\begin{equation*}
\Phi_{s} \colon H^{q}(X, K_{X} \otimes F) 
\xrightarrow{\otimes s} 
H^{q}(X, K_{X} \otimes F^{m+1} )
\end{equation*}
is injective for any $q$. 
Here $K_{X}$ denotes the canonical bundle of $X$. 
\end{theo}

In his paper \cite{Eno90}, Enoki gave the following injectivity theorem.  
Koll$\rm{\acute{a}}$r's proof 
is based on the Hodge theory. 
On the other hand, 
Enoki's proof is based on the theory of harmonic integrals, 
which enables us to 
approach the injectivity theorem 
from the viewpoint of complex differential geometry.

\begin{theo}[\cite{Eno90}] \label{Eno}
Let $F$ be a semi-positive line bundle on a compact 
K\"ahler manifold $X$. 
Then the same conclusion as in Theorem \ref{Kol} holds. 
\end{theo}
A semi-ample line bundle is always semi-positive 
(namely, it admits a \lq \lq smooth" metric with semi-positive curvature), 
and thus Theorem \ref{Eno} leads to Theorem \ref{Kol}. 
The above results can be regarded as 
a generalization of the Kodaira vanishing theorem 
to semi-ample (semi-positive) line bundles.  
On the other hand, 
the Kodaira vanishing theorem can be generalized to 
the Nadel vanishing theorem 
by using singular metrics  with 
(strictly) positive curvature, which corresponds to 
the Kawamata-Viehweg vanishing theorem 
in algebraic geometry.  
Therefore, in the same direction, 
it is natural to generalize them  
to an injectivity theorem for singular metrics with semi-positive curvature. \vspace{0.2cm}
\vspace{0.2cm}
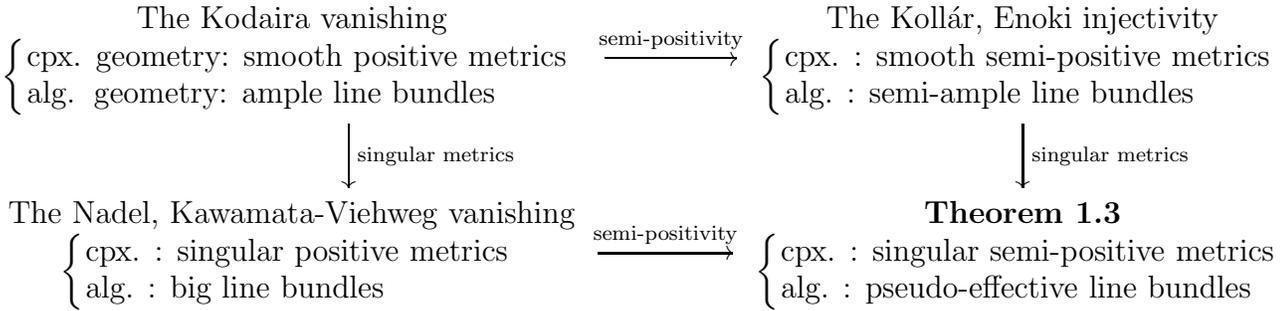
\[\xymatrixcolsep{4pc}\xymatrix{
\hspace{-1.5cm} \text{$
\begin{array}{cl}
\text{The Kodaira vanishing} \\
\left \{\hspace{-0.5cm} 
\begin{array}{cl}
&\text{cpx. geometry: smooth positive metrics}\\
&\text{alg. \,geometry: ample line bundles} 
\end{array}
\right.
\end{array}
$}
\ar[d]^-{\text{singular metrics}} \ar[r]^-{\text{semi-positivity}} &
\text{$\begin{array}{cl}
\text{The Koll\'ar, Enoki injectivity}\\
\left \{\hspace{-0.5cm} 
\begin{array}{cl}
&\text{cpx. : smooth semi-positive metrics}  \\
&\text{alg. : semi-ample line bundles} 
\end{array}
\right.
\end{array}
$}
\ar[d]^-{\text{singular metrics}}\\
\hspace{-1.5cm} 
\text{$
\begin{array}{cl}
\text{The Nadel, Kawamata-Viehweg vanishing}\\
\left \{\hspace{-0.5cm} 
\begin{array}{cl}
&\text{cpx. : singular positive metrics}\\
&\text{alg. : big line bundles} 
\end{array}
\right.
\end{array}
$} \ar[r]^-{\text{semi-positivity}}
&\text{$
\begin{array}{cl}
\text{{\bf{Theorem \ref{main}}}}\\
\left \{\hspace{-0.5cm}
\begin{array}{cl}
&\text{cpx. : singular semi-positive metrics}\\
&\text{alg. : pseudo-effective line bundles} 
\end{array}
\right. 
\end{array}
$}}
\]
\vspace{0.3cm}


The following theorem, which is the main result of this paper, 
is a common generalization of Koll\'ar's (Enoki's) 
injectivity theorem and the Nadel (Kawamata-Vieweg) vanishing theorem.  
Moreover Theorem \ref{main} is also a generalization of 
various results, for example, 
\cite{Eno90}, \cite{Fuj12-A}, \cite{Kol86}, 
\cite{Mat13-A}, \cite{Ohs04}, \cite{Tak97}, \cite{Tan71}. 
A (holomorphic) line bundle is said to be \textit{pseudo-effective} if 
it admits a \lq \lq singular" metric with semi-positive curvature, 
and thus Theorem \ref{main} 
can be seen as an injectivity theorem 
for pseudo-effective line bundles.

\begin{theo}[The Main Result]\label{main}
Let $F$ be a $($holomorphic$)$ line bundle on a compact K\"ahler manifold $X$ and 
$h$ be a singular $($hermitian$)$ metric with 
semi-positive curvature on $F$. 
Then, for a $($non-zero$)$ section $s$ of a positive multiple $F^{m}$  
satisfying $\sup_{X} |s|_{h^{m}} < \infty$, 
the multiplication map 
\begin{equation*}
\Phi_{s} \colon H^{q}(X, K_{X} \otimes F \otimes \I{h}) 
\xrightarrow{\otimes s} 
H^{q}(X, K_{X} \otimes F^{m+1} \otimes \I{h^{m+1}})
\end{equation*}
is $($well-defined and$)$ injective for any $q$. 
Here $\I{h}$ denotes the multiplier ideal sheaf 
associated to the singular metric $h$.  
\end{theo}

\begin{rem}\label{main-rem}
(1) We can show that the multiplication map from 
$H^{q}(X, K_{X} \otimes F^{\ell} \otimes \I{h^\ell})$ to 
$H^{q}(X, K_{X} \otimes F^{\ell+m} \otimes \I{h^{\ell+m}})$
is also injective for $\ell > 0$
by applying Theorem \ref{main} to $F^{\ell}$ 
and $s^{\ell} \in H^{0}(X, F^{\ell m})$. 
\vspace{0.1cm}\\
(2)
The multiplication map is well-defined 
thanks to the assumption $\sup_{X} |s|_{h^{m}} < \infty $. 
We can always apply this theorem to a pseudo-effective line bundle $F$ 
since a metric $h_{\min}$ with minimal singularities on $F$ 
satisfies $\sup_{X} |s|_{h_{\min}^{m}} < \infty $
for any section $s$ of $F^{m}$ (see \cite{Dem} for the definition of 
metrics with minimal singularities). 
\end{rem}

It is important to emphasize that 
a singular metric $h$ in our formulation may have 
transcendental (non-algebraic) singularities. 
To handle singular metrics with transcendental singularities, 
we have to take a more analytic approach 
to the cohomology groups with coefficients in 
$K_{X} \otimes F \otimes \I{h}$, 
which gives a generalization of techniques of 
\cite{Eno90}, \cite{Fuj12-A}, \cite{Mat13-A}, 
\cite{Mat13-B}, \cite{Tak97}.

Our formulation is motivated by the problem of extending 
sections from subvarieties to the ambient space. 
When we attempt to extend 
sections by the vanishing (injectivity) theorem or 
the Ohsawa-Takegoshi extension theorem, 
we need a suitable singular metric $h$. 
In many cases, the metric $h$ is constructed by 
taking the limit of suitable metrics $h_{m}$. 
For example, this strategy plays a crucial role
in the proof of the invariance of plurigenera or 
the extension theorem for pluri-canonical sections   
(see \cite{DHP13}, \cite{Pau07}, \cite{Siu98}). 
However it seems to be quite hard 
to investigate the regularity (smoothness) 
of the limit $h$, 
even if $h_{m}$ has algebraic singularities. 
Therefore it is worth for important applications 
to formulate Theorem \ref{main} for arbitrary singular metrics.

Thanks to this advantage, as applications of Theorem \ref{main}, 
we can obtain  
an injectivity theorem for 
nef and abundant line bundles 
(Corollary \ref{good}) 
and Nadel type vanishing theorems  
(Theorem \ref{gen}, Corollary \ref{main-co}),  
by considering metrics with minimal singularities 
(which do not always have algebraic singularities). 
Theorem \ref{gen} is a generalization of \cite{Mat13-A}, \cite{Mat13-B}. 
Moreover, we prove some extension theorems for pluri-canonical sections 
of log pairs motivated by the abundance conjecture in \cite{GM13}.

At the end of this section, 
we briefly explain the proof of Theorem \ref{main}. 
First we recall Enoki's proof of Theorem \ref{Eno} 
(the special case where $h$ is smooth). 
In this case, 
the cohomology group $H^{q}(X, K_{X} \otimes F)$ 
can be represented by the space 
of harmonic forms with respect to $h$
\begin{align*}
\mathcal{H}_{h}^{n, q}(F):= 
\{u \mid u 
\text{ is an $F$-valued $(n,q)$-form on } X 
\text{ with } \dbar u  = 0 \text{ and } \dbar^{*}_{h} u =0.\}, 
\end{align*}
where  $ \dbar^{*}_{h}$ is the adjoint operator of 
the $\dbar$-operator. 
For an arbitrary harmonic form $u \in \mathcal{H}_{h}^{n, q}(F)$,  
we can show that $su$ is also harmonic with respect to $h^{m+1}$  
from semi-positivity of the curvature of $h$. 
It implies that 
the multiplication map $\Phi_{s}$ 
induces the map from $\mathcal{H}_{h}^{n, q}(F)$ to 
$\mathcal{H}_{h^{m+1}}^{n, q}(F^{m+1})$, and then 
the injectivity is obvious. 
This method heavily depends on semi-positivity of the curvature.


In our situation, 
we can not directly use the theory of harmonic integrals 
since a given singular metric $h$ may have 
transcendental singularities. 
For this reason, in Step \ref{St1}, 
we first approximate the metric $h$ 
by singular metrics 
$\{ h_{\e} \}_{\e>0}$ that are smooth 
on a Zariski open set $Y$. 
Then a given cohomology class can be represented 
by the associated harmonic form  
$u_{\e}$ with respect to $h_{\e}$ on $Y$. 
We want to show that $s u_{\e}$ is harmonic 
by the same argument as in Enoki's proof. 
However, the same argument fails  
since the curvature of $h_{\e}$ is no longer semi-positive. 
Therefore, 
we investigate the asymptotic behavior of 
the harmonic form $u_{\e}$. 
This asymptotic analysis contains a new ingredient. 
In Step \ref{St2}, by generalizing Enoki's method, we prove that 
the $L^{2}$-norm 
$\| \dbar^{*}_{h_{\e}^{m+1}} su_{\e}\|$ converges to zero 
as $\e$ tends to zero. 
In Step \ref{St3}, we construct a solution $v_{\e}$ of 
the $\dbar$-equation $\dbar v_{\e} = su_{\e}$ 
such that   
the $L^{2}$-norm $\| v_{\e} \|$  
is uniformly bounded,  
by using the $\rm{\check{C}}$ech complex. 
The above arguments yield    
\begin{equation*}
\| su_{\e} \| ^{2} = 
\lla su_{\e}, \dbar v_{\e} \rra
\leq \| \dbar^{*}_{h_{\e}^{m+1}} su_{\e}\|
\| v_{\e} \| \to 0 \quad {\text{as }} \e \to 0.
\end{equation*}
In Step \ref{St4}, from these observations, we prove that 
$u_{\e}$ weakly converges to zero and this completes the proof.

This paper is organized as follows\,$:$  
In Section \ref{S2}, we summarize the fundamental facts used in this paper. 
We prove  the main result in Section \ref{S3} 
and give several applications of the main result in Section \ref{S4}. 
Compared to \cite{Mat13-B}, the crucial technique established in this paper 
is to construct a solution $v_{\e}$ of 
the $\dbar$-equation $\dbar v_{\e} = su_{\e}$ 
with uniformly bounded $L^{2}$-norm, 
by applying the De Rham-Weil isomorphism between the 
$\dbar$-cohomology and the $\rm{\check{C}}$ech cohomology. 
In Section \ref{S5}, we explain this construction 
after we study the topology of the $\rm{\check{C}}$ech complex induced 
by the local $L^2$-norms of singular metrics.  
This technique is rather complicated, 
but it seems to have more applications.

\subsection*{Acknowledgments}
The author wishes to express his gratitude to Professor Junyan Cao for 
fruitful discussions on the topology of the space of cochains, 
to Professor J\'anos Koll\'ar for a question concerning Remark \ref{main-rem},  
and to Professor Osamu Fujino 
for giving useful comments and 
teaching him the argument of Remark \ref{main-rem} (1).  
He would also like to express his gratitude to the referee for 
helpful comments and suggestions on the exposition of the paper. 
He is supported by the Grant-in-Aid for 
Young Scientists (B) $\sharp$25800051 from JSPS.


\section{Preliminaries}\label{S2}

In this section, 
we summarize the fundamental results 
for the proof of the main result.
Unless otherwise mentioned, $X$ denotes a compact K\"ahler manifold of dimension $n$ 
and $F$ denotes a (holomorphic) line bundle on $X$.

\subsection{Singular metrics and multiplier ideal sheaves}\label{S2-1}

We first recall the definition 
of singular metrics and curvatures. 
Fix a smooth (hermitian) metric $g$ on $F$. 

\begin{defi}\label{s-met}
{(Singular metrics and curvatures).}
(1) For an $L^{1}$-function $\varphi$ on $X$, 
the  metric $h$ defined by 
$
h:= g e^{-2\varphi} 
$
is called a \textit{singular hermitian metric} on $F$. 
Further $\varphi$ is called the \textit{weight} of $h$ 
with respect to the fixed smooth metric $g$. 
\vspace{0.1cm} \\ 
(2) The {\textit{curvature}} 
$\sqrt{-1} \Theta_{h}(F)$ 
of $h$ is defined by   
$$
\sqrt{-1} \Theta_{h}(F) = \sqrt{-1} \Theta_{g}(F)+ 2\deldel \varphi, 
$$
where $  \sqrt{-1} \Theta_{g}(F)$ is 
the Chern curvature of $g$. 
\end{defi}
In this paper, the singular hermitian metric is often written simply as the singular metric. 
The Levi form $\deldel \varphi$ is taken in the sense of 
distributions, 
and thus the curvature is a $(1,1)$-current 
but not always a smooth $(1,1)$-form. 
The curvature $\sqrt{-1} \Theta_{h}(F)$ of $h$ 
is said to be {\textit{semi-positive}} if 
$\sqrt{-1} \Theta_{h}(F) \geq 0$ in the sense of currents. 
When a singular metric $h$ satisfies $\sqrt{-1} \Theta_{h}(F) \geq \gamma$  
for some smooth $(1,1)$-form $\gamma$, 
the weight $\varphi$ of $h$ becomes 
a quasi-plurisubharmonic (quasi-psh for short) function.  
In particular $\varphi$ is upper semi-continuous, 
and thus it is bounded above. 

\begin{defi}
{(Multiplier ideal sheaves).}
Let $h$ be a singular metric on $F$ such that 
$\sqrt{-1} \Theta_{h}(F) \geq \gamma$  
for some smooth $(1,1)$-form $\gamma$ on $X$. 
Then the ideal sheaf $\I{h}$ defined 
to be  
\begin{equation*}
\I{h}(U):= \I{\varphi}(U):= \{f \in \mathcal{O}_{X}(U)\ \big|\ 
|f|e^{-\varphi} \in L^{2}_{\rm{loc}}(U) \}
\end{equation*}
for every open set $U \subset X$ is called 
the \textit{multiplier ideal sheaf} associated to $h$. 
\end{defi}

\subsection{Equisingular approximations}\label{S2-2}
In the proof, 
we apply the equisingular approximation to a given singular metric. 
In this subsection, 
we reformulate \cite[Theorem 2.3]{DPS01} 
with our notation and give an additional property. 

\begin{theo}[{\cite[Theorem 2.3]{DPS01}}]\label{equi}
Let $\omega$ be a positive $(1,1)$-form on $X$ and 
$h$ be a singular metric on $F$ with semi-positive curvature. 
Then there exist  singular 
metrics $\{h_{\e} \}_{1\gg \e>0}$ on $F$ with  the 
following properties\,$:$
\begin{itemize}
\item[(a)] $h_{\e}$ is smooth on $X \setminus Z_{\e}$, where $Z_{\e}$ is a subvariety on $X$.
\item[(b)]$h_{\e_{2}} \leq h_{\e_{1}} \leq h$ holds for any 
$0< \e_{1} < \e_{2} $.
\item[(c)]$\I{h}= \I{h_{\e}}$.
\item[(d)]$\sqrt{-1} \Theta_{h_{\e}}(F) \geq -\e \omega$. 
\end{itemize}
Moreover, if the set 
$\{x \in X \mid \nu(h, x) >0 \}$ is contained in 
a subvariety $Z$, then we can add the property 
that $Z_{\e} $ is contained in $Z$ for every $\e > 0$.  
Here $\nu(h, x)$ is the Lelong number at $x$ 
of the weight of $h$. 
\end{theo}
\begin{proof}
By applying \cite[Theorem 2.3]{DPS01} to the weight function $\varphi$ of $h$, 
we obtain quasi-psh functions 
$\varphi _{\nu}$ with equisingularities. 
If we take a sufficiently large $\nu=\nu(\e)$ for a given $\e>0$,
the metric $h_{\e}$ defined by $h_{\e}:=g e^{- 2 \varphi_{\nu(\e)}}$ 
satisfies properties (a), (b), (c), (d). 

The latter conclusion follows from the proof in \cite{DPS01}. 
We shortly see this fact 
by using the notation in \cite{DPS01}. 
In their proof, they locally approximate $\varphi$ by 
$\varphi_{\e, \nu, j}$ with logarithmic pole. 
From inequality (2.5) in \cite{DPS01}, 
the Lelong number of 
$\varphi_{\e, \nu, j}$ is 
less than or equal to that of $\varphi$. 
It follows that   
$\varphi_{\e, \nu, j}$ is smooth on $X \setminus Z$  
since $\varphi_{\e, \nu, j}$ has a logarithmic pole. 
Since $\varphi _{\nu}$ is obtained from 
Richberg's regularization of 
the supremum of 
these functions (see around (2.10)),   
we obtain the latter conclusion. 
\end{proof}

\subsection{The theory of harmonic integrals}\label{S2-3}
In this subsection, 
we recall the $L^2$-space of differential forms and the theory of harmonic integrals. 
Throughout this subsection, 
let $Y$ be a (not necessarily compact) 
complex manifold with a positive $(1,1)$-form $\ome$ 
and $E$ be a (holomorphic) line bundle  on $Y$ with a smooth metric $h$. 
In the proof of Theorem \ref{main}, 
the manifold $Y$ is a Zariski open set of $X$ and 
$E$ is the restriction of $F$ to $Y$.

For $E$-valued $(p,q)$-forms $u$ and $v$,  
the point-wise inner product 
$\langle u, v\rangle _{h, \ome}$
can be defined, and 
the (global) inner product 
$\lla u, v \rra  _{h, \ome}$ 
can also be defined by
\begin{equation*}
\lla u, v \rra  _{h, \ome}:=
\int_{Y} 
\langle u, v\rangle _{h, \ome}\, dV_{\ome}, 
\end{equation*}
where $dV_{\ome}:= \ome^{n}/n!$ and $n$ is the dimension of $Y$. 
The Chern connection $D_{h}$ on $E$ determined by the holomorphic structure and 
the hermitian metric $h$ can be written as 
$D_{h} = D'_{h} + \dbar$ with the $(1,0)$-connection $D'_{h}$ and the 
$(0,1)$-connection $\dbar$ (the $\dbar$-operator). 
The connections $D'_{h}$ and $\dbar$ 
are regarded as a densely defined closed operator on 
the $L^{2}$-space $L_{(2)}^{p, q}(Y, E)_{h, \ome}$ defined by 
\begin{equation*}
L_{(2)}^{p, q}(Y, E)_{h, \ome}:= 
\{u \mid u \text{ is an }E\text{-valued }(p, q)\text{-form with } 
\|u \|_{h, \ome}< \infty. \}. 
\end{equation*}
The formal adjoints $D'^{*}_{h}$ and $ \dbar^{*}_{h}$ agree with 
the Hilbert space adjoints in the sense of Von Neumann 
if $\ome$ is a \textit{complete} form on $Y$ 
(see \cite[(3.2) Theorem in ChapterV\hspace{-.1em}I\hspace{-.1em}I\hspace{-.1em}I]{Dem-book}). 
The following proposition can be obtained from   
the Bochner-Kodaira-Nakano identity and the density lemma.

\begin{prop}\label{Nak}
Let $ \ome$ be a complete K\"ahler form on $Y$ and 
$h$ be a smooth metric on $E$ such that  
$\sqrt{-1}\Theta_{h}(E) \geq  - C \ome$ for some constant $C>0$. 
Then, for every $u \in {\rm{Dom}}\, \dbar_{h}^{*} \cap {\rm{Dom}}\, \overline{\partial} 
\subset L_{(2)}^{n, q}(Y, E)_{h, \ome}$, 
the following equality holds\,$:$  
\begin{equation*}
\| \dbar_{h}^{*}u \|_{h, \ome}^{2} + 
\|\overline{\partial} u \|_{h, \ome}^{2} =  
\| D_{h}'^{*}u  \|_{h, \ome}^{2}+\lla \sqrt{-1}\Theta_{h}(E)\Lambda_{\ome} u, u
\rra_{h, \ome}.   
\end{equation*}   
Here $\Lambda_{\ome}$ 
denotes the adjoint of the wedge product $\ome \wedge \bullet$. 
\end{prop}

The following lemmas are proved by straightforward computations. 
For reader's convenience, we give a proof of Lemma \ref{rho}.

\begin{lemm}\label{cano}
Let $\omega$ and  $\ome$ be positive $(1,1)$-forms with $\ome \geq \omega$.  
If $u$ is an $(n,q)$-form, 
then $|u|^{2}_{\ome}\, dV_{\ome} \leq |u|^{2}_{\omega}\,  dV_{\omega} $. 
Further, if $u$ is an $(n,0)$-form, 
then $|u|^{2}_{\ome}\,  dV_{\ome} = |u|^{2}_{\omega}\,  dV_{\omega}$.
\end{lemm}

\begin{lemm}\label{rho}
Let $\omega$ be a positive $(1,1)$-form and $u$, $v$ be 
differential forms. 
\begin{itemize}
\item[(1)] There exists a positive 
constant $C$ $($depending only on the degree of $u$, $v$$)$ 
such that $|u \wedge v|_{\omega} \leq C |u|_{\omega}|v|_{\omega}$.
\item[(2)] If $\ome$ is a positive $(1,1)$-form  with $\ome \geq \omega$, 
then we have $|u|^2_{\ome} \leq |u|^2_{\omega}$. 
In particular, we have $|u \wedge v|_{\ome} \leq C |u|_{\ome}|v|_{\omega}$. 
\end{itemize}
\end{lemm}

\begin{proof}[Proof of Lemma \ref{rho}]
For a given point $x$, we choose  
a local coordinate $(z_{1}, z_{2}, \dots, z_{n})$ such that 
\begin{align*}
\omega = \frac{\sqrt{-1}}{2} \sum_{j=1}^{n} 
 dz_{j} \wedge d\overline{z_{j}}
 \quad \text{and} \quad 
\ome = \frac{\sqrt{-1}}{2} \sum_{j=1}^{n} 
\lambda_{j} dz_{j} \wedge d\overline{z_{j}} 
\quad  {\rm{ at}}\ x. 
\end{align*}
When the differential forms $u$ and $v$ are written as 
$u=\sum_{I, J}u_{I,J}dz_{I} \wedge d\overline{z}_{J}$ 
and 
$v=\sum_{K, L}v_{K,L}dz_{K} \wedge d\overline{z}_{L}$ 
in terms of this coordinate, 
it is easy to see that 
$$|u|^2_{\omega}=\sum_{I, J}|u_{I,J}|^{2} 
\quad \text{and} \quad  
|u|^2_{\ome}=\sum_{I, J}|u_{I,J}|^{2} 
\frac{1}{\prod_{(i,j) \in (I, J)} \lambda_{i} \lambda_{j}}
\quad  {\rm{ at}}\ x.$$ 
Here $I$, $J$, $K$, $L$  are ordered multi-indices 
and $dz_{I}:=dz_{i_{1}}\wedge dz_{i_{2}} \wedge \cdots \wedge dz_{i_{p}}$ 
for $I=\{ i_{1}<i_{2}< \cdots <i_p \}$. 
The second claim follows from $\lambda_{i} \geq 1$. 
The inequalities $|u_{I,J}| \leq |u|_{\omega}$ 
and $|v_{K,L}| \leq |v|_{\omega}$ yield 
\begin{align*}
|u \wedge v |_{\omega} 
& =
\big| \sum_{I, J, K, L} 
u_{I,J}\, v_{K,L}\, dz_{I} \wedge d\overline{z}_{J}
\wedge dz_{K} \wedge d\overline{z}_{L} \, \big|_{\omega} \\
& \leq \sum_{I, J, K, L} |u_{I,J}|  |v_{K,L}| \leq \sum_{I, J, K, L} 
 |u|_{\omega} |v|_{\omega}=C |u|_{\omega} |v|_{\omega}.   
\end{align*}
Here $C$ is a constant depending only on the degree of $u$, $v$.
\end{proof}

\subsection{Fr\'echet spaces}\label{S2-4}
In this subsection, for reader's convenience, 
we see that 
the following theorem leads to Proposition \ref{Fred}.

\begin{theo}[The open mapping theorem]
Let $\pi: D \to E$ be a linear map between Fr\'echet spaces 
$D$ and $E$. 
If $\pi$ is continuous and surjective, 
then $\pi$ is an open map. 
\end{theo}

\begin{prop}\label{Fred}
Let $\pi: D \to E$ be a continuous linear map between Fr\'echet spaces $D$ and $E$.
If the cokernel of $\pi$ is finite dimensional, 
then the image ${\rm{Im}}\, \pi$ of $\pi$ is closed in $E$.
\end{prop}
\begin{proof}
We first consider the case where $\pi: D \to E$ is injective. 
We take a finite dimensional subspace $E_{1}$ of $E$ 
such that 
the quotient map $p:E_1 \rightarrow E/ {\rm{Im}}\, \pi$ 
is isomorphic, 
and  consider a continuous map
$\pi_1:  D \oplus E_1 \rightarrow E$ defined to be  
$\pi_{1}(d, e):=\pi(d)+e$ for 
every $(d, e) \in D \oplus E_1$.
Since $\pi_1$ is surjective (and injective) and continuous,  
the inverse map 
$\pi_1^{-1}: E \rightarrow  D \oplus E_1 $ is also continuous 
by the open mapping theorem.  
By composing $\pi_1^{-1}$ with the second projection  
$D \oplus E_1  \rightarrow E_1$,  we obtain 
the continuous map $\pi_2: E\rightarrow E_1$. 
It is easy to see that the kernel of $\pi_2$ 
agrees with the image of $\pi$, 
which implies that  the image of $\pi$ is closed.
When $\pi: D \to E$ is not injective, 
by considering the linear map $\overline{\pi}:D/ {\rm{Ker}}\,\pi \to E$,  
we can obtain the conclusion.  
\end{proof}

\section{Proof of the main result}\label{S3}
In this section, we give a proof of the  main result. 
The proof is based on a technical combination of 
the theory of harmonic integrals 
and the $L^{2}$-method for the $\dbar$-equation.
\begin{theo}[=Theorem \ref{main}]\label{main-main}
Let $F$ be a line bundle on a compact K\"ahler manifold $X$ and 
$h$ be a singular metric with 
semi-positive curvature on $F$. 
Then, for a $($non-zero$)$ section $s$ of a positive multiple $F^{m}$  
satisfying $\sup_{X} |s|_{h^{m}} < \infty$, 
the multiplication map 
\begin{equation*}
\Phi_{s} \colon H^{q}(X, K_{X} \otimes F \otimes \I{h}) 
\xrightarrow{\otimes s} 
H^{q}(X, K_{X} \otimes F^{m+1} \otimes \I{h^{m+1}})
\end{equation*}
is $($well-defined and$)$ injective for any $q$. 
\end{theo}
\begin{proof}[Proof of Theorem \ref{main-main}]
The case $q=0$ is obvious, 
and thus we assume $q>0$. 
The proof can be divided into four steps. 
In Step \ref{St1}, we approximate a given singular metric $h$ by singular metrics 
$\{h_{\e} \}_{\e>0}$ that are smooth on a Zariski open set. 
In this step, we fix the notation to apply the theory of 
harmonic integrals and explain the sketch of the proof. 
For a given cohomology class in $H^{q}(X, K_{X} \otimes F \otimes \I{h})$ 
that goes to zero in $ H^{q}(X, K_{X} \otimes F^{m+1} \otimes \I{h^{m+1}})$ 
by $\Phi_{s}$, 
we take the associated harmonic form $u_{\e}$ 
with respect to $h_{\e}$. 
In Step \ref{St2}, we study the asymptotic behavior of 
$u_{\e}$ and $s u_{\e}$ as $\e$ tends to zero. 
In Step \ref{St3}, we construct a suitable solution $v_{\e}$ of 
the $\dbar$-equation $\dbar v_{\e} = s u_{\e}$. 
In Step \ref{St4}, we show that $u_{\e}$ converges to zero in a suitable sense.

\begin{step}[Equisingular approximation of $h$]\label{St1}
Throughout the proof, let $\omega$ be a K\"ahler form on $X$. 
For the proof, 
we want to apply the theory of harmonic integrals, 
but a given singular metric $h$ may not be smooth. 
For this reason, we approximate $h$ by singular metrics 
$\{ h_{\e} \}_{\e>0}$ that are smooth on a Zariski open set. 
By Theorem \ref{equi}, we obtain  singular 
metrics $\{ h_{\e} \}_{\e>0}$ on $F$ with the following properties\,$:$
\begin{itemize}
\item[(a)] $h_{\e}$ is smooth on $X \setminus Z_{\e}$, where $Z_{\e}$ is a subvariety on $X$.
\item[(b)]$h_{\e_{2}} \leq h_{\e_{1}} \leq h$ holds for any 
$0< \e_{1} \leq \e_{2} $.
\item[(c)]$\I{h}= \I{h_{\e}}$.
\item[(d)]$\sqrt{-1} \Theta_{h_{\e}}(F) \geq -\e \omega$. 
\end{itemize}
For the weight function $\varphi$ 
(resp. $\varphi_{\e}$) of the singular metric $h$ (resp. $h_{\e}$) with 
respect to a smooth metric $g$, 
we may assume $\varphi_{\e} \leq 0$ by adding a constant, 
since $\varphi_{\e}$ is bounded above on $X$. 
Hence we have 
\begin{equation*}
g \leq h_{\e}=g e^{-2\varphi_{\e}} \leq h=g e^{-2\varphi}.
\end{equation*}
Since the point-wise norm $|s|_{h^{m}}$ is bounded on $X$, 
there exists a constant $C$ such that 
$\log |s| \leq m \varphi + C$, 
where $s$ is locally regarded as a holomorphic function under a local trivialization of $F$. 
It implies that the Lelong number of $m \varphi$ is less than or equal to 
that of $\log |s|$. 
In particular, the set 
$\{x \in X \mid \nu(h, x)>0 \}$ 
is contained in the subvariety $Z$ defined by 
$Z:=\{x \in X \mid s(x)=0 \}$,  
and thus  
we may assume a stronger property than property (a), namely
\begin{itemize}
\item[(e)] $h_{\e}$ is smooth on $Y:=X \setminus Z$, 
where $Z:=\{x \in X \mid s(x)=0 \}$. 
\end{itemize}

Now we construct a \lq \lq complete" K\"ahler form on $Y$ 
with suitable potential function. 
Take a quasi-psh function $\psi$ on $X$ such that 
$\psi$ has a logarithmic pole along $Z$ 
and $\psi$ is smooth on $Y$.
Since the function $\psi$ is bounded above on $X$, 
we may assume $\psi \leq -e$ by adding a constant.
We define the $(1,1)$-form $\ome$ on $Y$ by 
\begin{equation*}
\ome:= \ell \omega + \deldel \Psi, 
\end{equation*}
where 
$\ell$ is a positive number and $\Psi:=1/\log(-\psi)$. 
Then we can show that    
the $(1,1)$-form $\ome$ satisfies the following properties 
for a sufficiently large $\ell >0$\,$:$ 
\begin{itemize}
\item[(A)] $\ome$ is a complete K\"ahler form on $Y$.
\item[(B)] $\Psi$ is bounded on $X$. 
\item[(C)] $\ome \geq \omega $.
\end{itemize} 
Indeed, properties (B), (C) follow   
from the definition of $\Psi$, $\ome$ and property (A) follows 
from straightforward computations.     
See \cite[Lemma 3.1]{Fuj12-A} for the precise proof of property (A). 
In the proof, 
we mainly consider $F$-valued differential forms on $Y$ (not $X$) and 
the $L^2$-norm with respect to $h_{\e}$ and $\ome$ (not $h$ and $\omega$).

Let $L_{(2)}^{n, q}(Y, F)_{h_{\e}, \ome}$ be the $L^2$-space of $F$-valued $(n,q)$-forms 
$u$ on $Y$ with respect to the inner product $\|\bullet \|_{h_\e, \ome}$
 defined by 
\begin{equation*}
\|u \|^{2}_{h_\e, \ome}:= \int_{Y} 
|u |^{2}_{h_{\e}, \ome}\, dV_{\ome}. 
\end{equation*}
Then, by Proposition \ref{closed}, 
we obtain the following orthogonal decomposition\,$:$ 
\begin{equation*}
L_{(2)}^{n, q}(Y, F)_{h_{\e}, \ome}
=
{\rm{Im}}\,\dbar
\oplus
\mathcal{H}_{h_{\e}, \ome}^{n, q}(F)
\oplus {\rm{Im}}\, \dbar^{*}_{h_{\e}}.  
\end{equation*}
As explained in subsection \ref{S2-3}, 
the operator $\dbar^{*}_{h_{\e}}$ (resp. $D'^{*}_{h_{\e}}$) 
denotes the formal adjoint of the densely defined closed operator   
$\dbar$ (resp. $D'_{h_{\e}}$), and they agree with 
the Hilbert space adjoints since $\ome$ is complete. 
(See \cite[\S3 Chapter V\hspace{-.1em}I\hspace{-.1em}I\hspace{-.1em}I]{Dem-book}
for a comparison of the formal adjoint and the Hilbert space adjoint.) 
Strictly speaking, the $\dbar$-operator also depends on $h_{\e}$, $\ome$ 
since the domain and range of $\dbar$ depend on them. 
We will write it simply $\dbar$ when no confusion can arise.
Here $\mathcal{H}_{h_{\e}, \ome}^{n, q}(F)$ is  
the space of harmonic forms with respect to 
$h_{\e}$ and $\ome$, namely 
\begin{equation*}
\mathcal{H}_{h_{\e}, \ome}^{n, q}(F)= 
\{u   \mid u  
\text{ is an } F\text{-valued } (n,q)\text{-form with  }
\dbar u=0 \text{ and } \dbar^{*}_{h_{\e}}u=0.    \}. 
\end{equation*}
A harmonic form in $ \mathcal{H}_{h_{\e}, \ome}^{n, q}(F)$ 
is smooth by elliptic regularity
(for example see 
\cite[(3.2) Theorem Chapter V\hspace{-.1em}I\hspace{-.1em}I\hspace{-.1em}I]{Dem-book}). 
These results seem to be known to specialists.  
The precise proof for them can be found in 
\cite{Dem-book}, \cite[Claim 1]{Fuj12-A}, and Section \ref{S5}.

It follows that  
$|u|^{2}_{h_{\e}, \ome}\, dV_{\ome} \leq 
|u|^{2}_{h_{\e}, \omega}\, dV_{\omega}$
for an $F$-valued $(n,q)$-form $u$
from Lemma \ref{cano} and property (C), 
which leads to the inequality $\|u \|_{h_{\e}, \ome}\leq 
\|u \|_{h_{\e}, \omega}$. 
From this inequality and property (b) of $h_{\e}$,
we obtain 
\begin{equation}\label{ine}
\|u \|_{h_{\e}, \ome} \leq 
\|u \|_{h_{\e}, \omega} \leq 
\|u \|_{h, \omega}
\end{equation}
for an $F$-valued $(n,q)$-form $u$. 
These inequalities play a crucial role in the proof. 
In this paper  
$\|\bullet \|_{\ome}$ denotes the $L^2$-norm on $Y$ (not $X$) 
and $\|\bullet \|_{\omega}$ denotes the $L^2$-norm on $X$ (not $Y$) 
if otherwise mentioned. 
Strictly speaking $\|u \|_{h_{\e}, \ome}$ is the norm of the restriction of $u$ to $Y$, 
but we will omit the notation of the restriction.

For the $L^2$-space $L^{n,q}_{(2)}(X, F)_{h, \omega}$ 
of $F$-valued $(n,q)$-forms on $X$ with respect to the inner product $\|\bullet \|_{h, \omega}$,  
we have the standard De Rham-Weil isomorphism 
$$
H^{q}(X, K_{X} \otimes F \otimes \I{h}) \cong 
\frac{{\rm{Ker}}\, \dbar}{{\rm{Im}}\, \dbar} \text{ of } 
L^{n,q}_{(2)}(X, F)_{h, \omega}, 
$$
where the right hand side is the $\dbar$-cohomology group 
defined by the closed operator $\dbar$ between 
$L^2$-spaces $L^{n,\bullet}_{(2)}(X, F)_{h, \omega}$. 
By this isomorphism, 
we can represent a given cohomology class  
by an $F$-valued 
$(n, q)$-form $u$ with $\|u \|_{h, \omega} < \infty$. 
In order to prove that 
the multiplication map $\Phi_{s}$ is injective, 
we assume that the cohomology class of $su$ 
is zero in 
$H^{q}(X, K_{X}\otimes F^{m+1}\otimes \I{h^{m+1}})$. 
Our final goal is to show that 
the cohomology class of $u$ is actually zero, 
that is, $u \in {\rm{Im}}\, \dbar \subset L^{n,q}_{(2)}(X, F)_{h, \omega}$.

It follows that 
$u \in L_{(2)}^{n, q}(Y, F)_{h_{\e}, \ome}$ for every $\e > 0$
from inequality (\ref{ine}). 
By the above orthogonal decomposition, 
there exist  
$u_{\e} \in \mathcal{H}_{h_{\e}, \ome}^{n, q}(F)$ and 
$w_{\e} \in {\rm{Dom}}\, \dbar \subset  L_{(2)}^{n,q-1}(Y, F)_{h_{\e}, \ome}$ such that 
\begin{equation*}
u=\dbar w_{\e} + u_{\e}. 
\end{equation*}
Note that  
the component of ${\rm{Im}}\, \dbar^{*}_{h_{\e}}$ is zero 
since $u$ is $\dbar$-closed.

At the end of this step, we explain  
the strategy of the proof. 
In Step \ref{St2}, we show that 
$\| \dbar^{*}_{h^{m+1}_{\e}} s u_{\e} \|_{h^{m+1}_{\e}, \ome}$ 
converges to zero 
as  $\e$ tends to zero. 
We have already known that there is a solution $v_{\e}$ of the $\dbar$-equation 
$\dbar v_{\e} = s u_{\e}$ since 
the cohomology class of $su$ is assumed to be zero. 
However, for our goal, we need $L^{2}$-estimates for $v_{\e}$. 
In Step \ref{St3}, 
we construct a solution $v_{\e}$ of the $\dbar$-equation   
$\dbar v_{\e} = s u_{\e}$ such that  
the norm $\| v_{\e} \|_{h^{m+1}_{\e}, \ome}$ 
is uniformly bounded. 
By Step \ref{St2} and Step \ref{St3}, we can obtain that 
\begin{equation*}
\|su_{\e} \|^{2}_{h^{m+1}_{\e}, \ome} \leq 
\| \dbar^{*}_{h^{m+1}_{\e}} s u_{\e} \|_{h^{m+1}_{\e}, \ome}
\| v_{\e} \|_{h^{m+1}_{\e}, \ome}\to 0 \text{\quad as } \e \to 0.
\end{equation*}
In Step \ref{St4}, from this convergence, we prove that 
$u_{\e}$ converges to zero  
in a suitable sense and this completes the proof. 
\end{step} 
\begin{step}[A generalization of Enoki's argument for the injectivity theorem]\label{St2}
The aim of this step is to prove the following proposition, 
which can be seen as a generalization of Enoki's proof of Theorem \ref{Eno}. 

\begin{prop} \label{D''}
As $\e$ tends to zero, the norm 
$\| \dbar^{*}_{h^{m+1}_{\e}} s u_{\e} \|_{h^{m+1}_{\e}, \ome}$ converges  
to zero. 
\end{prop}
\begin{proof}[Proof of Proposition \ref{D''}]
The key to prove the proposition is 
the following inequalities\,$:$ 
\begin{equation}\label{ine2}
\|u_{\e} \|_{h_{\e}, \ome} 
\leq \|u \|_{h_{\e}, \ome} 
\leq \|u \|_{h, \omega}<\infty.  
\end{equation}
The first inequality follows from the definition of  $u_{\e}$ 
and  the second inequality follows from 
inequality (\ref{ine}). 
The important point here is that the right hand side is independent of $\e$. 
By applying Proposition \ref{Nak} to $u_{\e}$, 
we obtain 
\begin{align}\label{B-eq}
0 = \lla \sqrt{-1}\Theta_{h_{\e}}(F)
\Lambda_{\ome} u_{\e}, u_{\e}
  \rra_{h_{\e}, \ome} + \|D'^{*}_{h_{\e}}u_{\e} \|^{2}_{h_{\e}, \ome} 
\end{align}
since $u_{\e}$ is harmonic with respect to $h_{\e}$ and $\ome$.
Let $A_{\e}$ be the first term and $B_{\e}$ be 
the second term of the right hand side of equality (\ref{B-eq}). 
We first show that the first term $A_{\e}$ and 
the second term $B_{\e}$ converge to zero. 
Let $g_{\e}$ be the integrand of $A_{\e}$, 
that is, 
\begin{equation*}
g_{\e}:= \langle  \sqrt{-1}\Theta_{h_{\e}}(F)
\Lambda_{\ome} u_{\e}, u_{\e}
 \rangle_{h_{\e}, \ome}. 
\end{equation*}
Then there exists a constant $C>0$ (independent of $\e$) 
such that
\begin{equation}\label{ine3}
g_{\e} \geq -\e C |u_{\e}|^{2}_{h_{\e}, \ome}. 
\end{equation}
This inequality follows from simple computations. 
Indeed, let $\lambda_{1}^{\e} \leq \lambda_{2}^{\e} \leq 
\dots \leq \lambda_{n}^{\e} $ be the 
eigenvalues of $\sqrt{-1}\Theta_{h_{\e}}(F)$ with respect to 
$\ome$. 
For every point $y \in Y$, there exists 
a local coordinate $(z_{1}, z_{2}, \dots, z_{n})$ 
centered at $y$ such that 
\begin{align*}
\sqrt{-1}\Theta_{h_{\e}}(F) =\frac{\sqrt{-1}}{2} \sum_{j=1}^{n} 
\lambda_{j}^{\e} dz_{j} \wedge d\overline{z_{j}}\quad \text{and} \quad 
\ome = \frac{\sqrt{-1}}{2} \sum_{j=1}^{n} 
 dz_{j} \wedge d\overline{z_{j}}
\quad {\rm{ at}}\ y. 
\end{align*}
When we locally write $u_{\e}$ as 
$u_{\e} =\sum_{|K|=q} u_{K}^{\e}\ dz_{1}\wedge \dots \wedge dz_{n} 
\wedge d\overline{z}_{K}$, 
we can easily see 
\begin{equation*}
g_{\e}= \sum_{|K|=q} 
\Big{(} \sum_{j \in K} \lambda_{j}^{\e} \Big{)} 
|u_{K}^{\e}|^{2}_{h_{\e}} 
\end{equation*}
by straightforward computations.  
On the other hand, from 
property (C) of $\ome$ and property (d) of $h_{\e}$, 
we have
$\sqrt{-1}\Theta_{h_{\e}}(F) 
\geq -\e \omega 
\geq -\e \ome$. 
It leads to $\lambda_{j}^{\e} \geq -\e$, 
and thus we obtain inequality (\ref{ine3}). 
From inequality (\ref{ine3}) and equality (\ref{B-eq}), 
we have  
\begin{align*}
0 \geq A_{\e} &= \int_{Y} g_{\e}\, dV_{\ome} \geq -\e C \int_{Y} |u_{\e}|^{2}_{h_{\e}, \ome}\, dV_{\ome} \geq -\e C \|u \|^{2}_{h, \omega}. 
\end{align*}
The last inequality follows from inequality (\ref{ine2}). 
Therefore $A_{\e}$ converges to zero. 
Further it follows that $B_{\e}$ also converges to zero 
from equality (\ref{B-eq}).

To apply Proposition \ref{Nak} to $su_{\e}$, 
we need to prove that $su_{\e} \in L_{(2)}^{n,q}(Y, F^{m+1}
)_{h^{m+1}_{\e}, \ome}$ and $su_{\e} \in {\rm{Dom}}\, \dbar^{*}_{h_{\e}^{m+1}}$. 
It can be proven that $su_{\e} \in {\rm{Dom}}\, \dbar^{*}_{h_{\e}^{m+1}}$ 
from \cite[Proposition\,2.2]{Mat15a}.
By the assumption, 
the point-wise norm $|s|_{h^{m}}$ with respect to $h^{m}$ is bounded. 
Further we have $|s|_{h_{\e}^{m}} \leq |s|_{h^{m}}$
from property (b) of $h_{\e}$. 
They imply   
\begin{equation*}
\|s u_{\e} \|_{h_{\e}^{m+1}, \ome} \leq 
\sup_{X} |s|_{h_{\e}^{m}}  \|u_{\e} \|_{h_{\e}, \ome}  \leq
\sup_{X} |s|_{h^{m}}  \|u \|_{h, \omega} < \infty. 
\end{equation*}
Hence we know $su_{\e} \in L_{(2)}^{n,q}(Y, F^{m+1}
)_{h^{m+1}_{\e}, \ome}$. 
Note that the right hand side is independent of $\e$. 
By applying Proposition \ref{Nak} to $su_{\e}$, 
we obtain 
\begin{align}\label{B-eq2}
&\| \dbar^{*}_{h_{\e}^{m+1}}
su_{\e} \|^{2}_{h^{m+1}_{\e}, \ome} \\
=& \lla \sqrt{-1}\Theta_{h^{m+1}_{\e}}(F^{m+1})\notag
\Lambda_{\ome} su_{\e}, su_{\e}
  \rra_{h^{m+1}_{\e}, \ome} +
\|D'^{*}_{h_{\e}^{m+1}}su_{\e} \|^{2}_{h^{m+1}_{\e}, \ome}.
\end{align}
Here we used $\dbar s u_{\e}=s \dbar u_{\e}=0$. 
From now on, we prove that  
the second term of the right hand side converges to zero. 
It is easy to see that  
$D'^{*}_{h^{m+1}_{\e}}su_{\e} =
s D'^{*}_{h_{\e}}u_{\e}$ holds  
since $s$ is a holomorphic section and 
$D'^{*}=-*\dbar*$, 
where $*$ is the Hodge star operator with respect to $\ome$. 
Therefore we have    
\begin{equation*}
\|D'^{*}_{h_{\e}^{m+1}}su_{\e} \|^{2}_{h^{m+1}_{\e}, \ome} \leq
\sup_{X}|s|^{2}_{h_{\e}^{m}}  \int_{Y}
|D'^{*}_{h_{\e}}u_{\e} |^{2}_{h_{\e}, \ome}\, dV_{\ome} 
\leq \sup_{X}|s|^{2}_{h^{m}} B_{\e}. 
\end{equation*}
Since $|s|^{2}_{h^{m}}$ is bounded and 
$B_{\e}$ converges to zero, 
the second term 
$\|D'^{*}_{h^{m+1}_{\e}}su_{\e} \|^2_{h^{m+1}_{\e}, \ome}$ converges to zero.

For the proof of the proposition, 
it remains to show that 
the first term of the right hand side of equality (\ref{B-eq2})
converges to zero. Now we investigate $A_{\e}$ in detail. 
By the definition of $A_{\e}$, we have
\begin{equation*}
A_{\e}= \int_{\{ g_{\e} \geq 0 \}} g_{\e}\, dV_{\ome} + 
\int_{\{ g_{\e} \leq 0 \}} g_{\e}\, dV_{\ome}. 
\end{equation*}
Let $A^{+}_{\e}$ be the first term and 
$A^{-}_{\e}$ be the second term of the right hand side. 
Then inequalities (\ref{ine2}) and (\ref{ine3}) lead to  
\begin{align*}
0 \geq A^{-}_{\e} &\geq 
-\e C \int_{\{ g_{\e} \leq 0 \}}|u_{\e}|^{2}_{h_{\e}, \ome}\, 
dV_{\ome} \geq -\e C \int_{Y}|u_{\e}|^{2}_{h_{\e}, \ome}\,  
dV_{\ome}\geq -\e C \| u \|^{2} _{h, \omega}. 
\end{align*}
Hence $A^{+}_{\e}$ and $A^{-}_{\e}$ converge to zero  
since $A_{\e} = A^{+}_{\e} + A^{-}_{\e}$ converges to zero. 
On the other hand, we have 
\begin{align*}
&\lla \sqrt{-1}\Theta_{h^{m+1}_{\e}}(F^{m+1})
\Lambda_{\ome} su_{\e}, su_{\e} \rra_{h^{m+1}_{\e}, \ome}\\
=& (m+1)\int_{Y} |s|^{2}_{h^{m}_{\e}} g_{\e}\, dV_{\ome} \\
=& (m+1) \Big\{ \int_{\{ g_{\e} \geq 0 \}} 
|s|^{2}_{h^{m}_{\e}} g_{\e}\, dV_{\ome} + 
\int_{\{ g_{\e} \leq 0 \}} 
|s|^{2}_{h^{m}_{\e}} g_{\e}\, dV_{\ome}\Big\}.
\end{align*}
Then it is easy to see the following inequalities\,$:$   
\begin{align*}
\bullet \quad 0 \leq \int_{\{ g_{\e} \geq 0 \}} 
|s|^{2}_{h^{m}_{\e}} g_{\e}\, dV_{\ome} &\leq  \sup_{X}|s|^{2}_{h^{m}_{\e}} 
\int_{\{ g_{\e} \geq 0 \}} 
g_{\e}\, dV_{\ome} \\ 
& \leq   \sup_{X}|s|^{2}_{h^{m}}\ A^{+}_{\e},  \\
\bullet \quad  0 \geq \int_{\{ g_{\e} \leq 0 \}} 
|s|^{2}_{h^{m}_{\e}} g_{\e}\, dV_{\ome} &\geq   
\sup_{X}|s|^{2}_{h^{m}_{\e}} \int_{\{ g_{\e} \leq 0 \}} 
 g_{\e}\, dV_{\ome} \\
&\geq   \sup_{X}|s|^{2}_{h^{m}} A_{\e}^{-}. 
\end{align*}
Therefore the right hand side of equality (\ref{B-eq2}) 
converges to zero. 
We obtain the conclusion of Proposition \ref{D''}. 
\end{proof}
\end{step}
\begin{step}[A construction of solutions of the $\dbar$-equation]\label{St3}
In this step, we prove Proposition \ref{sol} by using Theorem \ref{cru}. 
The proof of Theorem \ref{cru} is given in Section \ref{S5}.

\begin{prop}\label{sol-1}
There exists an $F$-valued $(n, q-1)$-form  $w_{\e}$ 
on $Y$ 
with the following properties\,$:$ 
\begin{equation*}
{\rm{(1)}}\hspace{0.2cm} \dbar w_{\e}=u-u_{\e}.   
\quad 
{\rm{(2)}}\hspace{0.2cm} \text{The norm }
\|w_{\e} \|_{h_{\e}, \ome} 
 \text{ is uniformly bounded}. 
\end{equation*}
\end{prop}
\begin{proof}
It is easy to see that $U_{\e}:=u-u_{\e} $ 
satisfies the assumptions of Theorem \ref{cru}. 
Indeed, it follows that 
$$
\|U_{\e}\|_{h_{\e}, \ome} \leq 
\|u \|_{h_{\e}, \ome} +\|u_{\e}\|_{h_{\e}, \ome}
\leq 2\|u\|_{h, \omega} < \infty
$$
from inequality (\ref{ine2}) and 
that $U_{\e}=u-u_{\e} \in {\rm{Im}}\, \dbar
\subset L^{n,q}_{(2)}(Y, F)_{h_{\e}, \ome}$ 
from the definition of $u_{\e}$. 
\end{proof}

\begin{prop}\label{sol}
There exists an $F^{m+1}$-valued $(n, q-1)$-form 
$v_{\e}$ 
on $Y $  
with the following properties\,$:$ 
\begin{equation*}
{\rm{(1)}}\hspace{0.2cm} \dbar v_{\e}=su_{\e}.   
\quad 
{\rm{(2)}}\hspace{0.2cm} \text{The norm }
\|v_{\e} \|_{h^{m+1}_{\e}, \ome} 
 \text{ is uniformly bounded}. 
\end{equation*}
\end{prop}
\begin{proof}
Since the cohomology class of $su$ is assumed to be zero 
in $H^{q}(X, K_{X}\otimes F^{m+1} \otimes \I{h^{m+1}})$,  
there exists an $F^{m+1}$-valued $(n, q-1)$-form $v$ 
such that $\dbar v =  su$ and 
$\|v \|_{h^{m+1}, \omega} < \infty$.  
If we take $w_{\e}$ satisfying the properties in Proposition \ref{sol-1} 
and put $v_{\e}:= -s w_{\e} + v$,  
then we have $\dbar v_{\e} = su_{\e}$. 
Further an easy computation yields  
\begin{align*}
\|v_{\e}\|_{h^{m+1}_{\e}, \ome} &\leq 
\|s w_{\e}\|_{h^{m+1}_{\e}, \ome} + 
\|v \|_{h^{m+1}_{\e}, \ome}\\
&\leq \sup_{X}|s|_{h^{m}} \|w_{\e} \|_{h_{\e}, \ome} + 
\|v \|_{h^{m+1}, \ome}. 
\end{align*}
By Lemma \ref{cano} and property (B), we have $\| v \|_{h^{m+1}, \ome}
\leq \|v \|_{h^{m+1}, \omega} < \infty$. 
Since the norm $\|w_{\e} \|_{h_{\e}, \ome}$ is uniformly bounded, 
the right hand side can be estimated by a constant independent of $\e$. 
\end{proof}
\end{step}

\begin{step}[Limit of the harmonic forms]\label{St4}
In this step, we show that $u_{\e} $ converges to zero in a suitable sense and 
this completes the proof. 
We first consider the following proposition obtained by Step \ref{St2} and Step \ref{St3}.

\begin{prop}\label{converge}
As $\e$ tends to zero, 
the norm $\| s u_{\e}\|_{h^{m+1}_{\e}, \ome}$ converges to zero. 
\end{prop}
\begin{proof}
For $v_{\e} \in L_{(2)}^{n, q-1}
(Y, F^{m+1})_{h^{m+1}_{\e}, \ome}$ satisfying the properties 
in Proposition \ref{sol},   
it is easy to see  
\begin{align*}
\| s u_{\e}\|_{h^{m+1}_{\e}, \ome}^{2}
&=\lla  s u_{\e}, \dbar v_{\e}  
 \rra_{h^{m+1}_{\e}, \ome} \\
&=
 \lla  \dbar^{*}_{h^{m+1}_{\e}}s u_{\e},v_{\e}  
 \rra_{h^{m+1}_{\e}, \ome} \\
&\leq 
\| \dbar^{*}_{h^{m+1}_{\e}} s u_{\e}\| _{h^{m+1}_{\e}, \ome}
\|v_{\e} \|_{h^{m+1}_{\e}, \ome}. 
\end{align*}
The norm   
$\|v_{\e} \|_{h^{m+1}_{\e}, \ome}$ is uniformly bounded by Proposition \ref{sol}.
On the other hand, the norm 
$\| \dbar^{*}_{h^{m+1}_{\e}} s u_{\e}\| _{h^{m+1}_{\e}, \ome}$ 
converges to zero by Proposition \ref{D''}. 
Hence the norm $\| s u_{\e}\|_{h^{m+1}_{\e}, \ome}$ converges to zero. 
\end{proof}

We want to take the limit of $u_{\e} \in L_{(2)}^{n, q}(Y, F)_{h_{\e}, \ome}$, 
but the $L^{2}$-space $L_{(2)}^{n, q}(Y, F)_{h_{\e}, \ome}$ depends on $\e$. 
For this reason we fix a small number $\e_{0}>0$ and 
consider the  fixed $L^{2}$-space $L_{(2)}^{n, q}(Y, F)_{h_{\e_{0}}, \ome}$. 
For every number $\e$ with $0< \e < \e_{0}$, 
we obtain 
\begin{equation*}
\|u_{\e} \|_{h_{\e_{0}}, \ome} \leq \|u_{\e} \|_{h_{\e}, \ome} 
\leq  \|u \|_{h, \omega} 
\end{equation*}
by property (b) of $h_{\e}$ and inequality (\ref{ine2}), 
which says that the norm of $u_{\e} $ with respect to 
$h_{\e_{0}}$ is 
uniformly bounded. 
In particular, 
there exists a subsequence of $\{ u_{\e} \}_{\e >0}$ 
that  converges to some 
$u_{0} \in L_{(2)}^{n, q}(Y, F)_{h_{\e_{0}}, \ome}$
with respect to the weak $L^{2}$-topology. 
For simplicity we continue to use the same notation $\{ u_{\e} \}_{\e >0}$ 
for this subsequence. 
The following proposition is proved by Proposition \ref{converge}. 

\begin{prop}\label{zero}
The weak limit $u_{0}$ of $\{ u_{\e} \}_{\e >0}$ 
in $L_{(2)}^{n, q}(Y, F)_{h_{\e_{0}}, \ome}$
is zero. 
\end{prop}
\begin{proof}

For every positive number $\delta>0$, 
we define the subset $Y_{\delta}$ by 
$Y_{\delta}:= \{x \in Y \mid  |s|^{2}_{h^{m}_{\e_{0}}} > \delta \text{ at } x. \}$. 
Since the weight $\varphi_{\e_{0}}$ 
of $h_{\e_{0}}$ is upper semi-continuous, 
the norm $|s|^{2}_{h^{m}_{\e_{0}}}$ is lower semi-continuous.   
In particular, the subset $Y_{\delta}$ is an open set of $Y$. 
A simple computation yields  
\begin{align*}
\| s u_{\e}  \|^{2}_{h^{m+1}_{\e}, \ome} 
\geq \| s u_{\e}  \|^{2}_{h^{m+1}_{\e_{0}}, \ome} 
\geq \int_{Y_{\delta}} |s|^{2}_{ h^{m}_{\e_{0}} } 
|u_{\e}|^{2}_{h_{\e_{0}}, \ome}\, dV_{\ome} 
\geq \delta  \int_{Y_{\delta}} 
|u_{\e}|^{2}_{h_{\e_{0}}, \ome}\, dV_{\ome} \geq 0
\end{align*}
for every $\delta>0$. 
Since the left hand side converges to zero by Proposition \ref{converge}, 
the norm $\|u_{\e} \|_{Y_{\delta}, h_{\e_{0}}, \ome}$ on $Y_{\delta}$ 
also converges to zero as $\e$ tends to zero.  
We can easily see that 
$u_{\e} |_{Y_{\delta}}$ 
converges to $u_{0} |_{Y_{\delta}}$ with respect to 
the weak $L^{2}$-topology in 
$ L_{(2)}^{n, q}(Y_{\delta}, F)_{h_{\e_{0}}, \ome}$. 
Here  $u_{\e} |_{Y_{\delta}}$ (resp. $u_{0} |_{Y_{\delta}}$) denotes 
the restriction of  $u_{\e}$ (resp. $u_{0}$) to $Y_{\delta}$. 
Indeed, for an arbitrary 
$w \in L_{(2)}^{n, q}(Y_{\delta}, F)_{h_{\e_{0}}, \ome}$, 
the inner product  
$\lla u_{\e} |_{Y_{\delta}}, w
\rra_{Y_{\delta}} 
= \lla u_{\e}, \widetilde{w} 
\rra_{Y} $
converges to 
$\lla u_{0}, \widetilde{w}
\rra_{Y} 
= \lla u_{0}|_{Y_{\delta}}, w 
\rra_{Y_{\delta}} $,  
where $\widetilde{w}$ denotes the zero extension of $w$ to 
$Y$. 
Since $u_{\e} |_{Y_{\delta}}$ weakly 
converges to $u_{0} |_{Y_{\delta}}$ and the norm is lower semi-continuous 
with respect to the weak $L^2$-topology, 
we obtain
\begin{equation*}
\|u_{0} |_{Y_{\delta}} \|_{Y_{\delta}, h_{\e_{0}}, \ome} 
\leq 
\liminf_{\e \to 0}\|u_{\e} |_{Y_{\delta}} \|_{Y_{\delta}, h_{\e_{0}},  \ome}=0. 
\end{equation*}
Hence $u_{0} |_{Y_{\delta}} = 0$ for every $\delta>0$. 
Since the union of 
$\{Y_{\delta} \}_{\delta >0}$ is equal to $Y=X \setminus Z$ by the definition of $Y_{\delta}$, 
the weak limit $u_{0} $ is zero on $Y$.
\end{proof}

By using Proposition \ref{zero}, 
we complete the proof of Theorem \ref{main-main}. 
By the definition of $u_{\e}$, 
we have 
\begin{equation*}
u = u_{\e} + \dbar w_{\e}. 
\end{equation*}
Proposition \ref{zero} asserts    
that $\dbar w_{\e}$ converges to $u$  
with respect to the weak $L^{2}$-topology in the fixed $L^2$-space. 
By the orthogonal decomposition, it is easy to show that 
$u$ is a $\dbar$-exact form in the fixed $L^2$-space 
(that is, $u \in {\rm{Im}}\, \dbar \subset L^{n,q}_{(2)}(Y, F)_{h_{\e_{0}}, \ome}$). 
Indeed, for every 
$w = w_{1} + \dbar^{*}_{h_{\e_{0}}} w_{2} \in
\mathcal{H}_{h_{\e_{0}}, \ome}^{n, q}(F)
\oplus {\rm{Im}}\, \dbar^{*}_{h_{\e_{0}}}$, we have 
$\lla u, w  \rra =\lim_{\e \to 0}\lla 
\dbar w_{\e}, w_{1} + \dbar^{*}_{h_{\e_{0}}}w_{2} \rra =0$.

From $u \in {\rm{Im}}\, \dbar \subset 
L_{(2)}^{n, q}(Y, F)_{h_{\e_{0}}, \ome}$ and  property (c), 
we can show that $u \in {\rm{Im}}\, \dbar \subset 
L_{(2)}^{n, q}(Y, F)_{{h}, \omega}$, 
which says that the cohomology class $\{u\}$ is zero. 
To clarify our argument, 
let $\dbar_{h, \omega}$ 
(resp. $\dbar_{h_{\e_{0}}, \ome}$) be 
the closed operator $\dbar$ 
between $L^{2}$-spaces $L^{n, \bullet}_{(2)}(X, F)_{h, \omega}$ 
(resp. $L^{n, \bullet}_{(2)}(Y, F)_{h_{\e_{0}}, \ome}$).
We consider the Dolbeault cohomology group and the de Rham-Weil isomorphism. 
Then we have the following commutative diagram\,$:$ 
\[\xymatrixcolsep{4pc}
   \xymatrix{
\dfrac{{\rm{Ker}}\, \dbar_{h, \omega}} {{\rm{Im}}\, \dbar_{h, \omega}} 
\ar[r]^-{j} \ar[d]^-{\cong}_-{{\overline{f}_{1}}} & \dfrac{{\rm{Ker}}\, \dbar_{h_{\e_{0}}, \ome}}{{\rm{Im}}\, \dbar_{h_{\e_{0}}, \ome}} \ar[d]^-{\cong}_-{{\overline{f}_{2}}}\\
\check{H}^{q}(X, K_{X}\otimes F \otimes \I{h})  & \check{H}^{q}(X, K_{X}\otimes F \otimes \I{h_{\e_{0}}}). \ar@{=}[l]
} \]
Here $j$ is the map induced by the natural map from 
$L^{n, \bullet}_{(2)}(X, F)_{h, \omega}$ to 
$L^{n, \bullet}_{(2)}(Y, F)_{h_{\e_{0}}, \ome}$, 
and $\overline{f}_{i}$ is the De Rham-Weil isomorphism to 
the $\rm{\check{C}}$ech cohomology group. 
(See Section \ref{S5} or \cite[Claim\,1]{Fuj12-A} for the construction of $\overline{f}_{2}$). 
The blew equality is obtained from $\I{h_{\e}}=\I{h}$. 
Here we essentially used property $(c)$.  
It follows that the cohomology class $\{u\}$ represented by 
$u \in L^{n, q}_{(2)}(X, F)_{h, \omega}$ goes to zero by $j$ 
from $u \in {\rm{Im}}\, \dbar \subset 
L_{(2)}^{n, q}(Y, F)_{h_{\e_{0}}, \ome}$. 
Therefore we can obtain that $u \in {\rm{Im}}\, \dbar \subset 
L_{(2)}^{n, q}(Y, F)_{{h}, \omega}$ by chasing the above diagram. 
\end{step}
\end{proof}

\section{Applications}\label{S4}
In this section, we give two corollaries of Theorem \ref{main} 
and their proof. 
One is an injectivity theorem for nef and abundant line bundles, 
and the other is a Nadel type vanishing theorem.  

It is reasonable to expect 
the same conclusion as in Theorem \ref{Kol} to hold for nef line bundles, 
but there exist counterexamples 
to the injectivity theorem for nef line bundles.  
However, it follows from \cite[Proposition\,2.1]{Kaw85} 
(cf. \cite{Nak85}, \cite[Corollary\,1]{Rus09}) 
that a metric $h_{\min}$ with minimal singularities on $F$ 
satisfies $\I{h_{\min}^{m}}= \mathcal{O}_{X}$ for any $m>0$  
if $F$ is nef and abundant 
(that is, the numerical dimension  
agrees with the Kodaira dimension).   
Therefore Theorem \ref{main} leads to the following corollary. 
(On projective varieties, 
a similar conclusion was proved in 
\cite{EP08} and \cite{EV92} by different methods.)  
It is worth to point out that Theorem \ref{Eno} is 
not sufficient to obtain Corollary \ref{good}. 
This is because,  
the above metric $h_{\min}$ is not smooth and 
does not always have algebraic singularities 
even if $F$ is nef and abundant  
(for example, see \cite[Example\,5.2]{Fuj12-C}).

\begin{cor}\label{good}
Let $F$ be a nef and abundant line bundle 
on a compact K\"ahler manifold $X$. 
Then the same conclusion as in Theorem \ref{Kol} holds.
That is, for a $($non-zero$)$ section $s$ of a positive multiple 
$F^{m}$ of the line bundle $F$, 
the multiplication map 
induced by the tensor product with $s$ 
\begin{equation*}
\Phi_{s}: H^{q}(X, K_{X} \otimes F) 
\xrightarrow{\otimes s} 
H^{q}(X, K_{X} \otimes F^{m+1} )
\end{equation*}
is injective for any $q$. 
\end{cor}

As another application, we can obtain 
a Nadel type vanishing (Theorem \ref{gen}), 
which leads to the following corollary.

\begin{cor}[{cf.\,\cite{Cao12}, \cite{Mat13-B}}]
\label{main-co}
Let $F$ be a line bundle on a smooth projective 
variety $X$ and  
$h_{\min}$ be 
a metric with minimal singularities on $F$.
Then  
\begin{equation*}
H^{q}(X, K_{X} \otimes F \otimes \I{h_{\min}}) = 0 
\hspace{0.4cm} {\text{for}}\ {\text{any}}\ q > n-\kappa(F).
\end{equation*}
Here $\kappa(F)$ denotes the Kodaira dimension of $F$. 
\end{cor}
This result is non-trivial  
even when the line bundle $F$ is big 
(that is, $\kappa(F)=n$). 
In his paper \cite{Cao12}, Cao proved 
the celebrated vanishing theorem for 
cohomology groups with coefficients in 
$K_{X}\otimes F \otimes \mathcal{I}_{+}(h)$. 
It is relatively easier to handle $\mathcal{I}_{+}(h)$ 
than $\mathcal{I}(h)$
(see \cite{DEL00} for the precise definition).  
If $h_{\min}$ has algebraic singularities, 
we can easily see that  $\mathcal{I}_{+}(h_{\min})$ agrees 
with $\mathcal{I}(h_{\min})$, 
but unfortunately $h_{\min}$ does not 
always have algebraic singularities. 
Thanks to Theorem \ref{main}, we can obtain Corollary \ref{main-co}  
without the assumption of algebraic singularities.

\begin{rem}
Three months after we finish writing our preprint, 
Guan and Zhou proved  
the strong openness conjecture in \cite{GZ13}. 
Further another proof was given by 
Hi$\rm{\hat{\underline{e}}}$p in \cite{Hie14}
and by Lempert in \cite{Lem14}. 
Although their celebrated result and Cao's theorem 
lead to Corollary \ref{main-co}, 
we believe that it is worth to display our techniques. 
This is because, our techniques are quite different
from them and 
give a new viewpoint to prove the vanishing theorem via the asymptotic vanishing theorem.  
\end{rem}



At the end of this section, 
we prove Theorem \ref{gen} by using  Theorem \ref{main}. 
First we give the following definition. 
\begin{defi} \label{bdd-Ko}
Let $F$ be a line bundle on a compact complex manifold $X$ and 
$h$ be a singular metric on $F$. \\
(1)
We define $H_{{\rm{bdd}}, h}^{0}(X, F)$ by 
the space of sections of $F$ with bounded norm with respect to $h$. 
That is, 
\begin{equation*}
H_{{\rm{bdd}}, h}^{0}(X, F): =
\{ s \in H(X, F) \mid \sup_{X} |s|_{h} < \infty \}.
\end{equation*}
\\
(2) The {\textit{generalized Kodaira dimension}} 
$\kappa_{\rm{bdd}}(F, h)$ of $(F,h)$ is defined to be $-\infty$ 
if $H_{{\rm{bdd}}, h^{m}}^{0}(X,F^{m})=0$ for any 
$m>0$. 
Otherwise, $\kappa_{\rm{bdd}}(F, h)$ is defined 
by 
\begin{equation*}
\kappa_{\rm{bdd}}(F, h): = 
\sup \{ k \in \mathbb{Z}\mid \limsup_{m \to \infty} 
\dim H_{{\rm{bdd}}, h^{m}}^{0}(X, F^{m})\big/ m^{k} > 0 \}. 
\end{equation*}
\end{defi}

If $h_{\min}$ is a metric with minimal singularities on $F$, 
the norm $|s|_{h_{\min}^{m}}$ 
is bounded on $X$ for any section $s \in H^{0}(X, F^{m})$. 
(For example see \cite{Dem} or \cite{Mat13-B}.) 
It implies that 
$H_{{\rm{bdd}}, h_{\min}^{m}}^{0}(X, F^{m})$ is isomorphic 
to  $H^{0}(X, F^{m})$ for every $m \geq 0$. 
In particular, $\kappa_{\rm{bdd}}(F, h_{\min})$ agrees with 
the usual Kodaira dimension $\kappa(F)$. 
Therefore the following theorem leads to  Corollary \ref{main-co}. 

\begin{theo}\label{gen}
Let $F$ be a line bundle 
on a smooth projective variety $X$ and 
$h$ be a singular metric  on $F$ with semi-positive curvature. 
Then 
\begin{equation*}
H^{q}(X, K_{X}\otimes F \otimes \I{h}) = 0
\hspace{0.4cm} {\text{for}}\ {\text{any}}\ 
q > n-\kappa_{\rm{bdd}}(F, h).
\end{equation*}
\end{theo}
\begin{proof}
For a contradiction, 
we assume that there exists a non-zero cohomology class 
$\alpha \in H^{q}(X, K_{X}\otimes F \otimes \I{h})$. 
If sections  
$\{s_{i}\}_{i=1}^{N}$ in 
$H_{{\rm{bdd}}, h^{m}}^{0}(X, F^{m})$ are 
linearly independent, 
then  $\{s_{i} \alpha \}_{i=1}^{N}$ 
are also  linearly independent in 
$ H^{q}(X, K_{X}\otimes F^{m+1} \otimes \I{h^{m+1}})$. 
Indeed, if 
$\sum_{i=1}^{N} c_{i} s_{i} \alpha =0$  
for some $c_{i} \in \mathbb{C}$, 
then we obtain $\sum_{i=1}^{N} c_{i} s_{i}=0$ by Theorem \ref{main}. 
Since $\{s_{i}\}_{i=1}^{N}$ are linearly independent, 
we have $c_{i}=0$ for every $i=1,2,\dots, N$. 
Therefore we obtain  
\begin{equation*}
\dim H_{{\rm{bdd}}, h^{m}}^{0}(X, F^{m}) \leq 
\dim H^{q}(X, K_{X}\otimes F^{m+1} \otimes \I{h^{m+1}}). 
\end{equation*}
On the other hand, by \cite[Theorem\,4.1]{Mat13-A}, 
we have the asymptotic vanishing theorem 
\begin{equation*}
\dim H^{q}(X, K_{X}\otimes F^{m} \otimes \I{h^{m}})= 
O(m^{n-q}) \quad \text{as } m \to \infty
\end{equation*}
for any $q\geq 0$ 
(cf. \cite[(6.18)\,Lemma]{Dem}). 
If $q > n-\kappa_{\rm{bdd}}(F, h)$, it is a contradiction.  
\end{proof}


\section{$\rm{\check{C}}$ech complex and De Rham-Weil isomorphism}\label{S5}
The aim of this section is to prove Theorem \ref{cru}, 
which gives solutions of the $\dbar$-equation with suitable $L^2$-estimates in Step \ref{St3} of Section \ref{S3}. 
The $\rm{\check{C}}$ech complex and the De Rham-Weil isomorphism play a crucial role 
in the proof of Theorem \ref{cru}.

\subsection{On the space of cochains}\label{S5-1}
In this subsection, for the proof of Theorem \ref{fre}, 
we study the space of cochains with the topology induced by 
the local $L^{2}$-norms with respect to singular metrics, 
which is used when we prove Theorem \ref{cru}. 
We first recall the following result on holomorphic functions, 
which can be proved by the division theorem. 
See \cite[Section D, Chapter I\hspace{-.1em}I]{GR65} for the proof. 

\begin{theo}[{\cite[Theorem 2, Section D, Chapter I\hspace{-.1em}I]{GR65}}] \label{division}
Let $G_{1}, G_{2}, \dots, G_{N}$ be holomorphic functions 
on an open set $B$ in $\mathbb{C}^{n}$. 
If holomorphic functions $\{G_{i}\}_{i=1}^{N}$ generate 
the stalk $\mathcal{I}_{p}$  at $p \in B$ of an ideal sheaf $\mathcal{I} \subset \mathcal{O}_{B}$, 
then there exist a neighborhood $L_{p} \Subset B$ of $p$  
and a constant $C_{p}>0$ with 
the following property\,$:$ \vspace{0.1cm}\\
\quad For every holomorphic function $F$ on $L_{p}$ 
whose germ at $p$ 
belongs to $\mathcal{I}_{p}$, 
there exist holomorphic functions $\{h_{j}\}_{j=1}^{N}$ on 
$L_{p}$ 
such that 
\begin{equation*}F= \sum_{j=1}^{N} h_{j} G_{j}
\quad \text{and} \quad
\sup_{L_{p}} |h_{j}| \leq C_{p} \sup_{L_{p}} |F|. 
\end{equation*}
\end{theo}
This theorem leads to the following lemma. 
In his paper \cite{Cao12}, Cao proved the former conclusion of 
the lemma when a quasi-psh function 
$\varphi$ has analytic singularities. 
For our purpose, we need a generalization 
of his result and the stronger conclusion (the latter 
conclusion of the lemma).

\begin{lemm}\label{key}
Let $\varphi$ be a quasi-psh function 
on an open set $B$ in $\mathbb{C}^{n}$ and 
$G_{1}, G_{2}, \dots, G_{N}$ be holomorphic functions on $B$ 
that generate the stalk of 
the multiplier ideal sheaf $\I{\varphi}$ at every point in $B$. 
Consider a sequence of holomorphic functions $\{f_{k}\}_{k=1}^{\infty}$ 
satisfying the following properties\,$:$ 
\begin{itemize}
\item[(1)] $f_{k}$ belongs to $H^{0}(B, \I{\varphi})$
$($that is, $|f_{k}| e^{-\varphi}$ is locally $L^{2}$-integrable 
on $B$$)$.
\item[(2)] $\{f_{k} \}_{k=1}^{\infty}$ uniformly converges to $f$ 
on every relatively compact set in $B$. 
\end{itemize}
Then the limit $f$  belongs to $H^{0}(B, \I{\varphi})$. 
Moreover, for every relatively compact set $K \Subset B$, 
the $($local$)$ $L^{2}$-norm  
\begin{equation*}
\int_{K} |f_{k} - f|^{2} e^{-2 \varphi} 
\end{equation*}
converges to zero  
as $k$ tends to infinity. 
\end{lemm}
\begin{proof}
For an arbitrary point $p \in B$, 
there exist 
a neighborhood 
$L_{p} \Subset B$ of $p$ and a positive constant $C_{p}$ 
with the property in Theorem \ref{division}. 
Since the germ of $f_{k}$ belongs to 
the stalk $\I{\varphi}_{p}$, 
there exist holomorphic functions $\{h_{k, j}\}_{j=1}^{N}$ on $L_{p}$ 
such that 
\begin{equation*}f_{k}= \sum_{j=1}^{N} h_{k, j} G_{j}
\quad \text{and} \quad
\sup_{L_{p}} |h_{k, j}| \leq C_{p} \sup_{L_{p}} |f_{k}|. 
\end{equation*}
The sup-norm $\sup_{L_{p}} |f_{k}|$ on $L_{p}$ is uniformly bounded 
by property (2). 
The above inequality implies that the sup-norm $\sup_{L_{p}} |h_{k, j}|$ 
is also uniformly bounded, 
and thus by Montel's theorem there exists a subsequence 
$\{h_{k_{\ell}, j} \}_{\ell=1}^{\infty}$ 
that uniformly converges to a holomorphic function $h_{j}$ 
on every relatively compact set in $L_{p}$. 
For every point $x$ in $L_{p}$ we have
\begin{align*}
f(x) = \lim_{\ell \to \infty} f_{k_{\ell}}(x) 
= \lim_{\ell \to \infty} \sum_{j=1}^{N} h_{k_{\ell}, j}(x) G_{j}(x)
= \sum_{j=1}^{N} h_{j}(x) G_{j}(x). 
\end{align*}
Therefore the germ of $f$ belongs to $\I{\varphi}_{p}$  
since the germ of $G_{j}$ belongs to 
$\I{\varphi}_{p}$.

Finally, we prove the latter conclusion. 
We have already known that the germ of $f_{k}-f$ 
belongs to $\I{\varphi}_{p}$.  
By Theorem \ref{division}, there exist a relatively compact set 
$L_{p} \Subset B$, a positive constant $C_{p}$,   
and holomorphic functions $\{g_{k,j}\}_{j=1}^{N}$ on $L_{p}$ such that  
\begin{equation*}
f_{k}-f = \sum_{j=1}^{N} g_{k,j} G_{j} \text{\quad and \quad} 
\sup_{L_{p}} |g_{k,j}|
\leq 
C_{p}\sup_{L_{p}} |f_{k}-f| \to 0.
\end{equation*}
On the other hand, an easy computation yields  
\begin{align*}
\int_{L_{p}} |f_{k} - f|^{2} e^{-2 \varphi} 
&\leq \int_{L_{p}} 
\Big(\sum_{j=1}^{N} |g_{k,j}|^{2}\Big) 
\Big(\sum_{j=1}^{N} |G_{j}|^{2}\Big) e^{-2 \varphi} \\
&\leq 
\Big(\sum_{j=1}^{N} \sup_{L_{p}}|g_{k,j}|^{2}\Big)
\int_{L_{p}} 
\sum_{j=1}^{N} |G_{j}|^{2} e^{-2 \varphi}. 
\end{align*} 
The right hand side converges to zero  
since 
the integral of $|G_{j}|^{2} e^{-2 \varphi}$ is finite and 
$g_{k, j}$ uniformly converges to zero on $L_{p}$. 
For a given relatively compact set $K \Subset B$, 
by taking a finite cover  $\{L_{p_{\nu}}\}_{\nu=1}^{m}$ of $K$,  
we can see
\begin{equation*}
\int_{K} |f_{k} - f|^{2} e^{-2 \varphi} \leq 
\sum_{\nu=1}^{m} \int_{L_{p_{\nu}}} |f_{k} - f|^{2} e^{-2 \varphi} \to 0.
\end{equation*}
This completes the proof. 
\end{proof}

To prove Theorem \ref{fre}, 
we recall the notation on the space of cochains. 
Let $G$ be a line bundle on a complex manifold $X$ and 
$h$ be a singular metric on $G$ satisfying $\sqrt{-1}\Theta_{h}(G) \geq \gamma$ 
for some smooth $(1,1)$-form $\gamma$. 
We take a Stein cover $\mathcal{U}:=\{B_{i}\}_{i \in I}$ of $X$ with the following properties\,$:$ 
\begin{itemize}
\item[$\bullet$] $G$ admits a local trivialization on $B_{i}$.
\item[$\bullet$] There are holomorphic functions on $B_{i}$ that  
generate the stalk of the multiplier ideal sheaf $\I{h}$ at every point in $B_{i}$. 
\end{itemize}
Note that we can take such an open cover 
since the multiplier ideal sheaf $\I{h}$ is coherent by a theorem of Nadel. 
Let $C^{q}(\mathcal{U}, G\otimes \I{h})$ be 
the space of $q$-cochains 
with coefficients in $G\otimes \I{h}$.
For a $q$-cochain $\alpha=\{\alpha_{i_{0}...i_{q}}\}_{i_{0}...i_{q}} \in 
C^{q}(\mathcal{U}, G \otimes \I{h})$, 
we often omit the notation of the subscript, such as \lq \lq$i_{0}...i_{q}$" 
and regard $\alpha_{i_{0}...i_{q}}$ as a holomorphic function 
under the trivialization of $G$ on $B_{i}$. 
The semi-norm  
$p_{K_{i_{0}...i_{q}}}(\bullet)$ is defined by 
\begin{equation*}
p_{K_{i_{0}...i_{q}}}(\alpha)^{2}:=
\int_{K_{i_{0}...i_{q}}} |\alpha_{i_{0}...i_{q}}|_{h}^{2} 
\end{equation*}
for a relatively compact set 
$K_{i_{0}...i_{q}} \Subset B_{i_{0}...i_{q}}:=B_{i_{0}}\cap \dots \cap B_{i_{q}}$. 
At the end of this section, 
we show that $C^{q}(\mathcal{U}, G\otimes  \I{h})$ is 
a Fr\'echet space with respect to these semi-norms.

\begin{theo}\label{fre}
In the above situation, 
the space of $q$-cochains 
$C^{q}(\mathcal{U}, G \otimes \I{h})$ is a Fr\'echet space. 
\end{theo}
\begin{proof}
For a given  Cauchy sequence $\big\{ \{\alpha_{k, i_{0}...i_{q}} \} \big\}_{k=1}^{\infty}$ 
in $C^{q}(\mathcal{U}, G\otimes \I{h})$,  
we put $\alpha_{k}:=\alpha_{k, i_{0}...i_{q}}$ and   
$B:=B_{i_{0}...i_{q}}$.  
Further we regard $\alpha_{k}$ as a holomorphic function on $B$. 
For the proof, it is sufficient to show that there exists 
a holomorphic function $\alpha$ on $B$ such that 
\begin{equation*}
\int_{K} |\alpha_{k}-\alpha|_{h}^{2}  \to 0
\end{equation*}
for every relatively compact set $K \Subset B$.

Since $\{\alpha_{k}\}_{k=1}^{\infty}$ is a Cauchy sequence 
with respect to the semi-norms, 
the $L^{2}$-norm $\int_{K} |\alpha_{k}|_{h}^{2}$ 
of $\alpha_{k}$ on $K$ is uniformly bounded. 
Since the local weight $\varphi$ of $h$ 
is quasi-psh, $\varphi$ is upper semi-continuous. 
In particular $\varphi$ is bounded above, 
and thus the $L^{2}$-norm $\int_{K} |\alpha|^{2}\, dV_{\omega}$ is also uniformly bounded. 
By Montel's theorem, 
there exists a subsequence $\{ \alpha_{k_{\ell}} \}_{\ell=1}^{\infty}$ 
of $\{ \alpha_{k} \}_{k=1}^{\infty}$ that uniformly 
converges to a holomorphic function $\alpha$ 
on every relatively compact set in $B$. 
Since this subsequence  $\{ \alpha_{k_{\ell}} \}_{\ell=1}^{\infty}$ 
satisfies the assumptions of Lemma \ref{key}, 
it can be shown that the limit $\alpha$ also belongs to $\I{h}$. 
Moreover,  we have   
\begin{equation*}
p_{K}(\alpha_{k_{\ell}} - \alpha)= 
\int_{K} |\alpha_{k_{\ell}} - \alpha|_{h}^{2} 
\to 0 
\end{equation*}
for every relatively compact set $K \Subset B$. 
Since $\{\alpha_{k}\}_{k=1}^{\infty}$ is a Cauchy sequence, 
the semi-norm $p_{K}(\alpha_{k} - \alpha)$ also converges to zero. 
\end{proof}

\subsection{De Rham-Weil isomorphisms}\label{S5-2}
In this subsection, we observe the construction of 
the De Rham-Weil isomorphism between 
the $\dbar$-cohomology and the $\rm{\check{C}}$ech cohomology in detail. 
The content of this subsection is essentially contained in \cite{Fuj12-A}. 

Let $\omega$ be a K\"ahler form on a compact K\"ahler manifold $X$ 
and $h$ be a singular metric on $F$ satisfying 
$\sqrt{-1}\Theta_{h}(F) \geq -a \omega$ for some constant $a>0$. 
Further let $Z$ be a subvariety on $X$ and 
let $\ome$ be a K\"ahler form on the Zariski open set $Y:=X \setminus Z$ 
with the following properties\,$:$
\begin{itemize}
\item[(B)] For every point $p$ in $X$, 
there exist an open neighborhood $B$ of $p$ and a bounded function $\Phi$ on $B$ such that $\ome =\deldel \Phi$ on $B \setminus Z$. 
\item[(C)] $\ome \geq \omega $.
\end{itemize} 
The important point is that $\ome$ locally admits 
a \lq \lq bounded" potential on a neighborhood of every point $p$ in $X$ (not $Y$). 
Note that the K\"ahler form $\ome$ constructed in Step \ref{St1} satisfies 
these properties. 
When we construct the De Rham-Weil isomorphism, 
we locally solve the $\dbar$-equation with $L^2$-estimate 
by using the following lemma.

\begin{lemm}[{cf.\,\cite[4.1 Th\'eor\`eme]{Dem82}}] \label{L2}
Under the same situation as above, 
we assume that $B$ is a Stein open set in $X$ with property $(B)$. 
Then, for an arbitrary 
$\alpha \in L^{n, q}_{(2)}(B\setminus Z, F)_{h, \ome}$ 
with $\dbar \alpha=0$, 
there exist $\beta \in L^{n, q-1}_{(2)}(B\setminus Z, F)_{h, \ome}$ 
and a positive constant $C$ $($depending only on $a$, $\Phi$, $q$$)$ 
such that 
\begin{align*}
 \dbar \beta &= \alpha, \\ 
 \int_{B\setminus Z} 
|\beta|^{2}_{h, \ome}\, dV_{\ome} 
& \leq C \int_{B\setminus Z} 
|\alpha|^{2}_{h, \ome}\, dV_{\ome}. 
\end{align*}
\end{lemm}

\begin{proof}
For a bounded function $\Phi$ on $B$ with $\ome=\deldel \Phi$, 
we define the metric $H$ on $F$ by 
$H:=h e^{-(1+a)\Phi}$. 
Then it follows that the curvature of $H$ satisfies 
\begin{align*}
\sqrt{-1}\Theta_{H}(F)&= \sqrt{-1}\Theta_{h}(F) + (1+a)\deldel \Phi  \geq -a \ome + (1+a)\ome  \geq  \ome 
\end{align*}
from  $\ome \geq \omega$ and 
$\sqrt{-1}\Theta_{h}(F) \geq -a\omega$.  
The $L^{2}$-norm $\|\alpha \|_{H, \ome}$ with respect to $H$ is finite 
since the function $\Phi $ is bounded and $\|\alpha \|_{h, \ome}$ is finite. 
We remark that $\ome$ is not a complete form on $B \setminus Z$, but 
$B \setminus Z$  admits a complete K\"ahler form. 
Therefore, from the standard $L^{2}$-method for the $\dbar$-equation 
(for example see \cite[4.1 Th\'eor\`eme]{Dem82}), 
we obtain a solution $\beta$ of the $\dbar$-equation 
$\dbar \beta =\alpha$ with  
\begin{equation*}
\|\beta \|^{2}_{H,\ome} 
\leq \frac{1}{q} \|\alpha \|^{2}_{H,\ome}.  
\end{equation*}
By putting $C_{1}:=\inf_{B} e^{-(a+1)\Phi}$ and 
$C_{2}:=\sup_{B} e^{-(a+1)\Phi}$, 
we have 
\begin{equation*}
C_{1} \|\beta \|^2_{h, \ome} \leq  \|\beta \|^2_{H, 
\ome}\quad  \text{and} \quad 
\|\alpha \|^2_{H, \ome} \leq C_{2} 
\|\alpha \|^2_{h, \ome}. 
\end{equation*}
These inequalities lead to the $L^{2}$-estimate in the lemma. 
\end{proof}

From now on, we fix a Stein finite cover 
$\mathcal{U}:=\{B_{i}\}_{i \in I}$ of $X$
such that $\ome$ admits a bounded potential function on $B_{i}$, 
and we consider the space of $q$-cochains 
$C^{q}(\mathcal{U}, K_{X}\otimes F \otimes \I{h})$ with coefficients in  
$K_{X}\otimes F \otimes \I{h}$  
with the topology induced by the semi-norms 
$p_{K_{i_{0}...i_{q}}}(\bullet)$ 
defined as follows: 
For every $\alpha=\{\alpha_{i_{0}...i_{q}}\} \in 
C^{q}(\mathcal{U}, K_{X} \otimes F \otimes \I{h})$ and 
a relatively compact set 
$K_{i_{0}...i_{q}} \Subset B_{i_{0}...i_{q}}$, 
the semi-norm  
$p_{K_{i_{0}...i_{q}}}(\alpha)$ of $\alpha$ is defined by 
\begin{equation*}
p_{K_{i_{0}...i_{q}}}(\alpha)^{2}:=
\int_{K_{i_{0}...i_{q}}} 
|\alpha_{i_{0}...i_{q}}|_{h, \omega}^{2}\, dV_{\omega}.  
\end{equation*}
This semi-norm is independent of 
the choice of $\omega$ by Lemma \ref{cano}. 
We remark that it is a Fr\'echet space 
(that is, it is complete with respect to these semi-norms)
by Theorem \ref{fre}. 
In this subsection, 
we observe the following De Rham-Weil isomorphism\,$:$ 
\begin{align*}
\dfrac{{\rm{Ker}}\, \dbar}{{\rm{Im}}\, \dbar} 
\text{ of } L^{n,q}_{(2)}(Y, F)_{h, \ome} 
\xrightarrow{\quad \cong \quad}
&\check{H}^{q}(\mathcal{U}, K_{X}\otimes F \otimes \I{h})\\
:=&
\dfrac{{\rm{Ker}}\, \delta}{{\rm{Im}}\, \delta} 
\text{ of } C^{q}(\mathcal{U}, K_{X}\otimes F \otimes \I{h}).  
\end{align*}
Here $\delta$ is the coboundary operator defined as follows\,: 
For every $q$-cochain 
$\{ \alpha_{i_{0}...i_{q}} \}_{i_{0}...i_{q}}$  
\begin{equation*}
\delta (\{ \alpha_{i_{0}...i_{q}} \}_{i_{0}...i_{q}} ):= 
\{ \sum_{\ell = 0}^{q+1} (-1)^{\ell} 
\alpha_{  i_{0}...\hat{i_{\ell}}...i_{q+1} }
\ |_{B_{i_{0}...i_{q+1}} } \}_{i_{0}...i_{q+1}},   
\end{equation*}  
where $B_{i_{0}...i_{q+1}} := B_{i_{0}}\cap...\cap B_{i_{q+1}}$.  
We will omit the notation of the restriction 
in the right hand side.

\begin{prop}\label{DWiso}
Under the same situation as above, 
there exist continuous maps    
\begin{align*}
f: {\rm{Ker}}\, \dbar \text{ in } L^{n,q}_{(2)}(Y, F)_{h, \ome} 
\rightarrow 
{\rm{Ker}}\, \delta \text{ in } 
C^{q}(\mathcal{U}, K_{X} \otimes F \otimes \I{h}) \\
g: {\rm{Ker}}\, \delta \text{ in } 
C^{q}(\mathcal{U}, K_{X} \otimes F \otimes \I{h})
\rightarrow 
{\rm{Ker}}\, \dbar \text{ in } L^{n,q}_{(2)}(Y, F)_{h, \ome} 
\end{align*}
satisfying the following properties\,$:$ 
\begin{itemize}
\item[$\bullet$] $f$ induces the isomorphism 
\begin{align*}
\overline{f} \colon \dfrac{{\rm{Ker}}\, \dbar}{{\rm{Im}}\, \dbar} 
\text{ of } L^{n,q}_{(2)}(Y, F)_{h, \ome}  
\xrightarrow{\quad \cong \quad }
\dfrac{{\rm{Ker}}\, \delta}{{\rm{Im}}\, \delta} 
\text{ of } C^{q}(\mathcal{U}, K_{X}\otimes F \otimes \I{h}).  
\end{align*}
\item[$\bullet$] $g$ induces the isomorphism 
\begin{align*}
\overline{g} \colon \dfrac{{\rm{Ker}}\, \delta}{{\rm{Im}}\, \delta} 
\text{ of } C^{q}(\mathcal{U}, K_{X}\otimes F \otimes \I{h})
\xrightarrow{\quad \cong \quad }
\dfrac{{\rm{Ker}}\, \dbar}{{\rm{Im}}\, \dbar} 
\text{ of } L^{n,q}_{(2)}(Y, F)_{h, \ome}.  
\end{align*}
\item[$\bullet$] $\overline{f}$ is the inverse map of $\overline{g}$. 
\end{itemize}
\end{prop}

\begin{proof}
We first define $f(U)\in {\rm{Ker}}\, \delta \subset C^{q}(\mathcal{U}, K_{X} \otimes F \otimes \I{h})$ 
for a given $U\in {\rm{Ker}}\, \dbar \subset L^{n,q}_{(2)}(Y, F)_{h, \ome}$.

For the $0$-cochain $\alpha^{0}:=\{ \alpha_{i_{0}} \}$ defined 
by $\alpha_{i_{0}}:=U|_{B_{i_{0}}\setminus Z}$, 
by applying Lemma \ref{L2} to $\alpha_{i_{0}}$, 
it is shown that 
the $\dbar$-equation $\dbar \beta_{i_{0}}= \alpha_{i_{0}}$ 
on ${B_{i_{0}}\setminus Z}$ has a solution.  
The solution $\beta_{i_{0}}$ whose $L^{2}$-norm 
$\|\beta_{i_{0}}\|_{h, \ome}$ is minimum among all solutions 
satisfies the $L^2$-estimate 
$\|\beta_{i_{0}}\|^2_{h, \ome} \leq 
C \|\alpha_{i_{0}}\|^2_{h, \ome}  \leq C \|U\|^2_{h, \ome}$. 
Here the constant $C$ does not depend on $h$ and $U$. 
In the proof $C$ denotes 
(possibly) different positive constants independent of $h$ and $U$.

For the $1$-cochain $\alpha^1$ defined by 
$\alpha^{1}:=\{\alpha_{i_{0}i_{1}} \}:=\delta \{\beta_{i_{0}} \}$, 
we have 
\begin{align*}
\| \alpha^{1} \|^{2}_{h, \ome}:=& 
\sum_{i_{0}, i_{1} \in I} \int_{B_{i_{0}i_{1}} \setminus Z} 
|\alpha_{i_{0}i_{1}}|^{2}_{h, \ome}\, dV_{\ome}\\
\leq& 
\sum_{i_{0}, i_{1} \in I} (\|\beta_{i_{0}}\|_{h, \ome}+
\|\beta_{i_{1}}\|_{h, \ome})^2
\leq C \|U\|^2_{h, \ome}
\end{align*}
for some constant $C$. 
Further we have 
$$\dbar \alpha^1=\dbar \delta \{\beta_{i_{0}} \}=
\delta \dbar \{\beta_{i_{0}} \}=\delta \{\alpha_{i_{0}}\}=0. 
$$
Therefore, 
by applying Lemma \ref{L2} to $\alpha^{1}$ again, 
we can take the solution of the $\dbar$-equation 
$\dbar \beta_{i_{0}i_{1}}= \alpha_{i_{0}i_{1}}$ on ${B_{i_{0}i_{1}}\setminus Z}$ 
whose $L^{2}$-norm $\|\beta_{i_{0}i_{1}}\|_{h, \ome}$ is minimum 
among all solutions. 
Note that we have 
$ 
\|\beta_{i_{0}i_{1}}\|^2_{h, \ome} \leq C \|\alpha_{i_{0}i_{1}}\|^2_{h, \ome}$. 
Similarly, by putting $\alpha^2:=\{\alpha_{i_{0}i_{1}i_{2}}\}:=\delta \{\beta_{i_{0}i_{1}}\}$, 
we can check  
$$
\| \alpha^{2} \|^{2}_{h, \ome}:= 
\sum_{i_{0}, i_{1} i_{2}\in I} \int_{B_{i_{0}i_{1}i_{2}} \setminus Z} 
|\alpha_{i_{0}i_{1}i_{2}}|^{2}_{h, \ome}\, dV_{\ome}
\leq C \|U\|^2_{h, \ome}. 
$$
By repeating this process, 
we can obtain the $k$-cochain $\{\beta_{i_{0}\dots i_{k}} \}$ 
with coefficients in the $F$-valued $(n, q-k-1)$-forms with 
the following equalities\,$:$
\[
  (*) \left\{ \quad
  \begin{array}{ll}
\vspace{0.2cm}
\dbar \beta_{i_{0}} &=U |_{B_{i_{0}}\setminus Z},  \\
\dbar \{ \beta_{i_{0}i_{1}} \}&=\delta  \{\beta_{i_{0}}\},  \\
\dbar \{ \beta_{i_{0}i_{1}i_{2}} \}&=\delta   \{\beta_{i_{0}i_{1}}\},  \\
 & \vdots   \\
\dbar \{ \beta_{i_{0}\dots i_{q-1}} \}&=\delta   \{\beta_{i_{0}\dots i_{q-2}}\}.
  \end{array} \right.
\]
It follows that 
$$
\|\beta_{i_{0}\dots i_{k}}\|^2_{h,\ome}\leq C \|U\|^2_{h,\ome} 
$$
from the construction. 
Now $\alpha^{q}:=\{\alpha_{i_{0}\dots i_{q}}\}:=\delta \{\beta_{i_{0}\dots i_{q-1}}\}$ is a $q$-cocycle with coefficients in  
the $F$-valued $(n,0)$-forms and 
satisfies 
$$
\dbar \alpha^{q}=\dbar \delta \{\beta_{i_{0}\dots i_{q-1}}\} 
=\delta  \dbar \{ \beta_{i_{0}\dots i_{q-1}}\} 
=\delta \delta \{ \beta_{i_{0}\dots i_{q-2}}\}=0. 
$$
In particular, $\alpha_{i_{0}...i_{q}}$ can be regarded as 
a holomorphic function with bounded $L^2$-norm  
since it is a $\dbar$-closed $F$-valued $(n,0)$-form and 
it satisfies  
$\|\alpha_{i_{0}...i_{q}}\|_{h, \omega}= 
\|\alpha_{i_{0}...i_{q}}\|_{h, \ome}<\infty$ by Lemma \ref{cano}. 
Then $\alpha_{i_{0}...i_{q}}$ can be extended  
from $B_{i_{0}\dots i_{q}} \setminus Z$ 
to the $\dbar$-closed $F$-valued $(n,0)$-form on $B_{i_{0}\dots i_{q}}$ 
by the Riemann extension theorem. 
Therefore it determines 
$\alpha^{q}\in {\rm{Ker}}\, \delta \subset C^{q}(\mathcal{U}, K_{X}\otimes F \otimes \I{h})$. 
We define $f(U)$ by $f(U):=\alpha^{q}=\delta \{ \beta_{i_{0}\dots i_{q-1}} \}$. 
It follows that $f$ is continuous 
from the construction of $f$ and the $L^{2}$-estimate 
$$
\| f(U) \|^{2}_{h, \ome} = 
\sum_{i_{0}\dots i_{q}\in I} \int_{B_{i_{0}\dots i_{q}} \setminus Z} 
|\alpha_{i_{0}\dots i_{q}}|^{2}_{h, \ome}\, dV_{\ome}
\leq C \|U\|^2_{h, \ome}.
$$

Next we define 
$g(\alpha^{q}) \in {\rm{Ker}}\, \dbar \subset L^{n,q}_{(2)}(Y, F)_{h, \ome}$ 
for a given 
$\alpha^{q}=\{\alpha_{i_{0}\dots i_{q}} \} \in {\rm{Ker}}\, \delta \subset C^{q}(\mathcal{U}, K_{X} \otimes F \otimes \I{h})$. 
For the $(q-1)$-cochain  
$\alpha^{q-1}:=\{\alpha_{i_{0}\dots i_{q-1}} \}$ defined by 
$$
\alpha_{i_{0}\dots i_{q-1}}:= 
\sum_{k \in I} \rho_{k}\alpha_{k i_{0}\dots i_{q-1}}\, , 
$$
we can easily check $\delta \alpha^{q-1}=\alpha^{q}$ 
from $\delta \alpha^{q}=0$, 
and thus we have 
$\delta \dbar \alpha^{q-1}= \dbar \delta \alpha^{q-1}= \dbar \alpha^{q}=0$. 
When 
we define $\alpha^{q-2}:=\{\alpha_{i_{0}\dots i_{q-2}} \}$ by 
$$
\alpha_{i_{0}\dots i_{q-2}}:= 
\sum_{k \in I} \rho_{k} \dbar \alpha_{k i_{0}\dots i_{q-2}}\, , 
$$
we can easily check $\delta \alpha^{q-2}=\dbar \alpha^{q-1}$ 
from $\delta\dbar \alpha^{q-1}=0$ again. 
By repeating this process, 
we obtain the $k$-cochain $\alpha^{k}$ with coefficients in  
the $F$-valued $(n,q-k-1)$-forms. 
Then $\dbar \alpha^{0}$ determines 
the $\dbar$-closed $F$-valued form globally defined on $X$ 
by $\delta \dbar \alpha^{0}=\dbar \delta \alpha^{0}=\dbar \dbar \alpha^{1}=0$. 
We define $g(\alpha^{q})$ by $g(\alpha^{q}):=\dbar \alpha^{0}$. 
The properties in Proposition \ref{DWiso} can be proved  
by the standard argument, and thus we omit it. 
\end{proof}

\begin{rem}\label{const}
(1) 
The map $g$ is linear, but $f$ is not linear. 
This is because the norm of $\beta_{1}+\beta_{2}$ is not necessarily minimum 
even if $\beta_{i}$ is the solution of $\beta_{i}=\dbar \alpha_{i}$ whose $L^2$-norm is minimum. 
The induced maps $\overline{f}$ and $\overline{g}$ are linear.\\
(2) 
For the proof of Theorem \ref{cru}, we remark that
\begin{align*}
& g(\alpha^{q})\\
=&\dbar \Bigg(  \sum_{k_{q}\in I}\rho_{k_{q}}
\dbar \bigg( \sum_{k_{q-1}\in I}\rho_{k_{q-1}} \cdots
\dbar \Big( \sum_{k_{3}\in I}\rho_{k_{3}}
\dbar \big( \sum_{k_{2}\in I}\rho_{k_{2}}
\dbar (\sum_{k_{1}\in I}\rho_{k_{1}} \alpha_{k_{1}...k_{q}i_{0}})
\big)\Big)\bigg) \Bigg)\\
=&
\sum_{k_{q}\in I} \dbar \rho_{k_{q}}\wedge
\sum_{k_{q-1}\in I} \dbar \rho_{k_{q-1}} \wedge \cdots
\sum_{k_{3}\in I} \dbar \rho_{k_{3}} \wedge
\sum_{k_{2}\in I} \dbar \rho_{k_{2}} \wedge
\dbar (\sum_{k_{1}\in I}\rho_{k_{1}} \alpha_{k_{1}...k_{q}i_{0}})
\end{align*}
holds on $B_{i_{0}}$ by the construction and the Leibnitz rule.
\end{rem}

Proposition \ref{DWiso} leads to the following lemma and proposition.

\begin{lemm}\label{lim}
Under the same situation as above, 
the space of $q$-cocycles 
$Z^{q}(\mathcal{U}, K_{X}\otimes F\otimes \I{h})
:={\rm{Ker}}\, \delta$
and the space of $q$-coboundaries 
$B^{q}(\mathcal{U}, K_{X}\otimes F \otimes \I{h})
:={\rm{Im}}\, \delta$ are 
closed subspaces in 
$C^{q}(\mathcal{U}, K_{X}\otimes F \otimes \I{h})$ 
$($in particular Fr\'echet spaces$)$.  
\end{lemm}
\begin{proof}
It is easy to check that 
the coboundary operator $\delta$  
from $C^{q}(\mathcal{U}, K_{X} \otimes 
F \otimes \I{h})$ to 
$C^{q+1}(\mathcal{U}, K_{X} \otimes 
F \otimes \I{h})$ is continuous. 
It implies that  
${\rm{Ker}}\, \delta$ 
is a closed subspace. 
Now we consider 
the following coboundary operator\,$:$ 
\begin{equation*}
\delta: C^{q-1}(\mathcal{U}, K_{X}\otimes F\otimes \I{h}) \to
Z^{q}(\mathcal{U}, K_{X}\otimes F\otimes \I{h}).
\end{equation*}
The cokernel of $\delta$ is isomorphic to 
$H^{p}(X, K_{X}\otimes F \otimes \I{h})$, 
whose dimension is finite.  
The open mapping theorem implies that 
${\rm{Im}}\, \delta$ is a closed subspace  
(see Proposition \ref{Fred}). 
\end{proof}

\begin{prop}\label{closed}
Under the same situation as above, 
the ranges ${\rm{Im}}\, \dbar \subset L^{n,q}_{(2)}(Y, F)_{h, \ome} $ 
and 
${\rm{Im}}\, \dbar^{*} \subset L^{n,q}_{(2)}(Y, F)_{h, \ome} $ 
are closed subspaces in  
$L^{n,q}_{(2)}(Y, F)_{h, \ome}$. 
\end{prop}
\begin{proof}
For a given sequence $\{U_{k}\}_{k=1}^{\infty}$ in ${\rm{Im}}\, \dbar$ 
that converges to some $U$, 
it is shown that $\overline{f}([U_{k}])$ converges to $\overline{f}([U])$ 
from Proposition \ref{DWiso}. 
Here $[\bullet]$ denotes the $\dbar$-cohomology class. 
We have $\overline{f}([U])=0$ from $\overline{f}([U_{k}])=0$ 
since the $\rm{\check{C}}$ech cohomology is a separated topological space by Lemma \ref{lim}. 
It follows that $U \in {\rm{Im}}\, \dbar$ since 
$\overline{f}$ is an isomorphism. 
It is shown that  
${\rm{Im}}\, \dbar^{*} $ 
is also closed from this fact (see \cite[Theorem\,1.1]{Hor}).
\end{proof}

\subsection{Proof of the key theorem}\label{S5-3}
In this subsection, we prove the following theorem\,$:$
\begin{theo}\label{cru}
Under the same situation as in subsection \ref{S5-2}, 
we consider a family of singular metrics $\{h_{\e}\}_{1 \gg \e >0}$ on $F$
with $h_{\e}\leq h$, $\I{h_{\e}}=\I{h}$, 
and $\sqrt{-1}\Theta_{h_{\e}}(F)\geq -\omega$. 
Then, for $U_{\e} \in {\rm{Im}}\,\dbar \subset 
L^{n,q}_{(2)}(Y, F)_{h_{\e}, \ome} $ such that   
the $L^{2}$-norm $\|U_{\e}\|_{h_{\e}, \ome}$ is uniformly bounded, 
there exists an $F$-valued $(n,q-1)$-form $V_{\e}$ 
with the following properties\,$:$
\begin{equation*}
{\rm{(1)}}\hspace{0.2cm} \dbar V_{\e}=U_{\e}.   
\quad 
{\rm{(2)}}\hspace{0.2cm} \text{The norm }
\|V_{\e}\|_{h_{\e}, \ome}  
 \text{ is uniformly bounded}. 
\end{equation*}
\end{theo}
\begin{proof}
The strategy of the proof is as follow\,$:$   
The main idea of the proof is to convert the $\dbar$-equation 
$\dbar V_{\e}=U_{\e}$
to the equation $\delta \gamma_{\e} =f_{\e}(U_{\e})$ of 
the coboundary operator $\delta$
in the space of cochains 
$C^{\bullet}(K_{X}\otimes F \otimes \I{h_{\e}})$, 
by using the $\rm{\check{C}}$ech complex and pursuing  
the De Rham-Weil isomorphism. 
Here $f_{\e}$ is the continuous map constructed for $h_{\e}$ 
in Proposition \ref{DWiso}. 
The important point is that 
$C^{\bullet}(K_{X}\otimes F \otimes \I{h_{\e}})$ is independent of 
$\e$ thanks to the property of $h_{\e}$
although the $L^{2}$-space 
$L^{n,\bullet}_{(2)}(Y, F)_{h_{\e}, \ome}$ depends on $\e$. 
Since $\|U_{\e}\|_{h_{\e}, \ome}$ is uniformly bounded, 
it is proven that $f_{\e}(U_{\e})$ converges to 
some $q$-coboundary in $C^{q}(K_{X}\otimes F \otimes \I{h})$ 
with the topology induced by 
the local $L^{2}$-norms with respect to $h$.
Further it is shown that 
the coboundary operator $\delta$ is an open map.
Then, by these observations, we can construct a solution $\gamma_{\e}$ of 
the equation $\delta \gamma_{\e} = f_{\e}(U_{\e})$ 
with uniformly bounded norm. 
Finally, by using a partition of unity 
(the map $g_{\e}$ constructed in Proposition \ref{DWiso}), 
we conversely 
construct $V_{\e} \in L^{n,q}_{(2)}(Y, F)_{h_{\e}, \ome}$
with the properties in Theorem \ref{cru}.

\vspace{0.2cm}
Fix a Stein finite cover $\mathcal{U}:=\{B_{i}\}_{i \in I}$ of $X$ 
such that $\ome$ admits a bounded potential function on $B_{i}$. 
By applying the argument in Proposition \ref{DWiso} to $U_{\e}$, 
we can obtain $\{\beta_{\e, i_{0}\dots i_{k}}\}$ satisfying equality ($*$). 
By the assumption of $h_{\e}$, 
we have 
\begin{align*}
\alpha^{q}_{\e}:=f_{\e}(U_{\e}):=\delta \{\beta_{\e, i_{0}\dots i_{q-1}}\} 
\in &C^{q}(\mathcal{U}, K_{X}\otimes F\otimes \I{h_{\e}})\\
=&C^{q}(\mathcal{U}, K_{X}\otimes F\otimes \I{h}).
\end{align*}
Then we prove the following claim. 
\begin{claim}\label{suq}
In the above situation, we have the following\,$:$
\begin{itemize} 
\item[$\bullet$] $f_{\e}(U_{\e}) \in {\rm{Im}}\, \delta \subset 
C^{q}(\mathcal{U}, K_{X}\otimes F\otimes \I{h})$ for every $\e>0$.
\item[$\bullet$] $\{f_{\e}(U_{\e}) \}_{\e>0}$ has a subsequence that  
converges to a $q$-cochain $\alpha^{q}$. 
\item[$\bullet$] The limit $\alpha^{q} \in {\rm{Im}}\, \delta \subset 
C^{q}(\mathcal{U}, K_{X}\otimes F\otimes \I{h})$. 
\end{itemize}
\end{claim}
\begin{proof}[Proof of Claim \ref{suq}]
It follows that $f_{\e}(U_{\e}) \in {\rm{Im}}\, \delta \subset 
C^{q}(\mathcal{U}, K_{X}\otimes F\otimes \I{h_{\e}})$ 
from the assumption $U_{\e} \in {\rm{Im}}\, \dbar 
\subset L_{(2)}^{n,q}(Y, F)_{h_{\e}, \ome}$ and Proposition \ref{DWiso}. 
By the assumption of $h_{\e}$, we obtain the first conclusion.

Now we prove that each component $\alpha_{\e, i_{0}\dots i_{q}}$ of $f_{\e}(U_{\e})$ 
has a subsequence that converges to some $F$-valued $(n,0)$-form. 
By the construction of $\alpha^{q}_{\e}=f_{\e}(U_{\e})=\delta(\{\beta_{\e, i_{0}\dots i_{q-1}}\})$, 
we have 
$$\|\alpha_{\e, i_{0}\dots i_{q}} \|^{2}_{h_{\e},\ome}
\leq \|f_{\e}(U_{\e}) \|^{2}_{h_{\e},\ome}\leq 
C \|U_{\e}\|^{2}_{h_{\e},\ome}. 
$$
Here the above constant $C$ does not depend on $U_{\e}, h_{\e}$, 
and thus the right hand side can be estimated by a constant 
independent of $\e$. 
In particular, $\alpha_{\e, i_{0}\dots i_{q}}$ 
can be regarded as a holomorphic function with uniformly bounded $L^2$-norm.  
(Note that $\alpha_{\e, i_{0}\dots i_{q}}$ is a $\dbar$-closed $F$-valued $(n,0)$-form). 
Then the sup-norm $\sup_{K_{i_{0}...i_{q}}} |\alpha_{\e, i_{0}...i_{q}}|$ 
is also uniformly bounded for every relatively compact set 
$K_{i_{0}...i_{q}} \Subset B_{i_{0}...i_{q}}$. 
Therefore, by Montel's theorem, we obtain 
a subsequence of $\{\alpha_{\e, i_{0}...i_{q}}\}_{\e>0}$ 
that uniformly converges to 
some $F$-valued $(n,0)$-form $\alpha_{i_{0}...i_{q}}$ 
on every relatively compact set in $B_{i_{0}...i_{q}}$. 
Lemma \ref{key} asserts that 
this subsequence converges to $\alpha_{i_{0}...i_{q}}$ 
with respect to the semi-norms 
$\{ p_{K_{i_{0}...i_{q}}}(\bullet)\}_{K_{i_{0}...i_{q}} \Subset B_{i_{0}...i_{q}}}$. 
Hence we can find 
a subsequence of $f_{\e}(U_{\e})$ 
that converges to some $q$-cochain 
$\alpha^{q}$ in $C^{q}(\mathcal{U}, K_{X} \otimes F \otimes \I{h})$.  
The latter conclusion follows from Lemma \ref{lim}
\end{proof}

We will construct a solution $\gamma_{\e}$ of 
the equation $\delta \gamma_{\e} = f_{\e}(U_{\e})$ 
with uniformly bounded norm. 
For simplicity we use the same notation $\{f_{\e}(U_{\e})\}_{\e>0}$ 
for the subsequence obtained in Claim \ref{suq}. 
Note that the space of $q$-coboundaries 
$B^{q}(\mathcal{U}, K_{X}\otimes F\otimes \I{h}):= {\rm{Im}}\,\delta$ 
is also a Fr\'echet space by Lemma \ref{lim}.
The following coboundary operator 
\begin{equation*}
\delta: C^{q-1}(\mathcal{U}, K_{X}\otimes F\otimes \I{h}) \to 
B^{q}(\mathcal{U}, K_{X}\otimes F\otimes \I{h})
\end{equation*}
is continuous and surjective linear map 
between Fr\'echet spaces, 
and thus  
this coboundary operator is an open map by the open mapping theorem.

By Claim \ref{suq}, there exists  
$\gamma \in C^{q-1}(\mathcal{U}, K_{X}\otimes F \otimes \I{h})$ 
such that $\delta \gamma = \alpha^{q}$. 
For a given family $K:=\{K_{i_{0}...i_{q-1}}\}$ 
of relatively compact sets 
$K_{i_{0}...i_{q-1}} \Subset B_{i_{0}...i_{q-1}}$, 
we define the open bounded neighborhood $\Delta_{K} $ of $\gamma$ in 
$C^{q-1}(\mathcal{U}, K_{X}\otimes F \otimes \I{h})$ 
by 
\begin{equation*}
\Delta_{K}:=\{ \beta \in C^{q-1}(\mathcal{U}, K_{X}\otimes F \otimes \I{h}) \mid \ p_{K_{i_{0}...i_{q-1}}}(\beta-\gamma) < 1 \}. 
\end{equation*}
Then $\delta (\Delta_{K})$ is an open neighborhood of the limit $\alpha^{q}$ in 
$B^{q}(\mathcal{U}, K_{X}\otimes F\otimes \I{h})$ 
by the above observation. 
Therefore $f_{\e}(U_{\e})$ belongs to 
$\delta (\Delta_{K})$ for a sufficiently small $\e>0$  
since $f_{\e}(U_{\e})$ converges to $\alpha^{q}$. 
By the definition of $\Delta_{K}$, 
we can obtain $\gamma_{\e}=:\{\gamma_{\e, i_{0}...i_{q-1}}\} \in 
C^{q-1}(\mathcal{U}, K_{X}\otimes F \otimes \I{h})$ 
such that 
\begin{align}\label{del-hol}
&\delta \gamma_{\e}= f_{\e}(U_{\e}), \\
\label{hol-est}
&p_{K_{i_{0}...i_{q-1}}}(\gamma_{\e})^{2}= 
\int_{K_{i_{0}...i_{q-1}}} 
|\gamma_{\e, i_{0}...i_{q-1}}|^{2}_{h, \omega}\ \omega^{n}
 \leq C_{K} 
\end{align}
for some positive constant $C_{K}$. 
The above constant $C_{K}$ depends on the choice of $K$, $\gamma$, 
but does not depend on $\e$.

We will construct an $F$-valued $(n,q)$-form $V_{\e}$ 
from $\gamma_{\e}$ and $f_{\e}(U_{\e})$ by using $g_{\e}$. 
The strategy is as follows\,$:$
It follows that $g_{\e}(\delta \gamma_{\e})=\dbar v_{\e}$ and 
$g_{\e}(f_{\e}(U_{\e}))=U_{\e}+\dbar \widetilde{v}_{\e}$ 
for some $v_{\e}$ and $\widetilde{v}_{\e}$ 
since $\overline{g}_{\e}$ gives the isomorphism in Proposition \ref{DWiso}. 
On the other hand, we have $U_{\e}=\dbar(v_{\e} - \widetilde{v}_{\e})$
by equality (\ref{del-hol}).  
Then we can concretely compute $v_{\e}$ and $\widetilde{v}_{\e}$ 
by using a partition of unity $\{\rho_{i}\}_{i\in I}$,  
and thus we obtain the $L^2$-estimate for them.

\begin{claim}\label{sol-2}
There exists an $F$-valued $(n, q-1)$-form $v_{\e}$ 
on $X$ satisfying the 
following properties\,$:$ 
\begin{equation*}
{\rm{(1)}}\hspace{0.2cm} \dbar v_{\e}=g_{\e}(\delta \gamma_{\e} ).   
\quad  
{\rm{(2)}}\hspace{0.2cm} \text{The norm }
\|v_{\e} \|_{h, \ome} 
 \text{ is uniformly bounded}. 
\end{equation*}
\end{claim}
\begin{proof}
We observe 
\begin{equation*}
\gamma_{\e,  k_{2}...k_{q}i_{0}} + 
\sum_{\ell = 2}^{q} (-1)^{\ell-1} 
\gamma_{\e,  k_{1}...\hat{k_{\ell}}...k_{q}i_{0} } + 
(-1)^{q} 
\gamma_{\e,  k_{1}...k_{q} }.
\end{equation*} 
and the construction of $g_{\e}$ (see Remark \ref{const}).  
\\
{\bf{[Argument 1]}}
\vspace{0.1cm}\\
First we consider the first term $\gamma_{\e,  k_{2}...k_{q}i_{0}}$. 
It is easy to see 
\begin{equation*}
\dbar  \sum_{k_{1}\in I}\rho_{k_{1}}
\gamma_{\e,  k_{2}...k_{q}i_{0}}= 
\dbar  \gamma_{\e,  k_{2}...k_{q}i_{0}}  =0   
\end{equation*}
since $\gamma_{\e,  k_{2}...k_{q}i_{0}}$ does not depend on $k_{1}$. 
Here we used $ \sum_{k_{1}\in I}\rho_{k_{1}}=1$. 
We can conclude  
that this term does not affect $g_{\e}(\delta \gamma_{\e})$ from Remark \ref{const}. 
\\
{\bf{[Argument 2]}}
\vspace{0.1cm}\\
Secondly we consider the second term $\gamma_{\e,  k_{1}...\hat{k_{\ell}}...k_{q}i_{0} }$. 
For an integer $\ell$ with $2 \leq \ell \leq q$, by the Leibniz rule,  
we can show 
\begin{align*}
&\sum_{k_{q}\in I}
\dbar \rho_{k_{q}} \wedge 
\cdots \wedge \sum_{k_{\ell}\in I} \dbar  \rho_{k_{\ell}} 
\wedge \sum_{k_{\ell-1}\in I} \dbar  \rho_{k_{\ell-1}}\wedge
\cdots  \wedge 
\sum_{k_{1}\in I} \dbar  (\rho_{k_{1}} 
\gamma_{\e,  k_{1}..\hat{k_{\ell}}..k_{q-1}i_{0} })\\
=&\sum_{k_{q}\in I}
\dbar \rho_{k_{q}} \wedge 
\cdots \wedge  \dbar \sum_{k_{\ell-1}\in I} \dbar  \rho_{k_{\ell-1}}\wedge
\cdots  \wedge 
\sum_{k_{1}\in I} \dbar  (\rho_{k_{1}} 
\gamma_{\e,  k_{1}..\hat{k_{\ell}}..k_{q-1}i_{0} })=0. 
\end{align*}
Here we used $\dbar\, \dbar = 0$ and $\sum_{k_{\ell}\in I}\rho_{k_{\ell}} =1$. 
Therefore the second term does not affect $g_{\e}(\delta \gamma_{\e})$. 
\\
{\bf{[Argument 3]}}\\
Finally we consider the third term 
$(-1)^q \gamma_{\e,  k_{1}...k_{q} }$. 
If $v_{\e}$ is defined by  
\begin{align*}
v_{\e}: =& (-1)^{q} \sum_{k_{1}, \dots, k_{q} \in I}
\rho_{k_{q}}\dbar \rho_{k_{q-1}}\wedge \dbar \rho_{k_{q-1}}\wedge
\cdots \wedge \dbar \rho_{k_{2}} \wedge \dbar (\rho_{k_{1}} \wedge \gamma_{\e,  k_{1}...k_{q} })\\
=& (-1)^{q} \sum_{k_{1}, \dots, k_{q} \in I}
\rho_{k_{q}}\dbar \rho_{k_{q-1}}\wedge \dbar \rho_{k_{q-1}}\wedge
\cdots \wedge \dbar \rho_{k_{2}} \wedge \dbar \rho_{k_{1}} \wedge \gamma_{\e,  k_{1}...k_{q} },  
\end{align*}
then $v_{\e}$ determines 
the $F$-valued $(n,q-1)$-form on $X$  
since $ \gamma_{\e,  k_{1}...k_{q}}$ 
is independent of $i_{0}$. 
The second equality follows from the Leibnitz rule and 
$\dbar \gamma_{\e,  k_{1}...k_{q}}=0$. 
We have $g_{\e}(\delta \gamma_{\e})= \dbar v_{\e}$ 
by the definition of $v_{\e}$ 
and Argument 1, 2. 
For the proof, it is sufficient to show that 
the norm $\|v_{\e} \|_{h, \ome}$ is uniformly bounded. 
When we define the $(0,q-1)$-form 
$\eta_{k_{1}...k_{q}}$ on $X$ 
by 
\begin{equation*}
\eta_{k_{1}...k_{q}}:=
\rho_{k_{q}}
\dbar \rho_{k_{q-1}} \wedge
\dbar \rho_{k_{q-2}} \wedge
\cdots \wedge
\dbar \rho_{k_{1}},    
\end{equation*}
we have 
 \begin{align*}
v_{\e} = \sum_{k_{1},\dots,k_{q} \in I} \eta_{k_{1}...k_{q}} \wedge \gamma_{\e,  k_{1}...k_{q} }.    
\end{align*} 
Since the support of $\eta_{k_{1}...k_{q}}$ is relatively compact in 
$B_{k_{1}...k_{q}}$, 
there exists $K:=\{K_{k_{1}...k_{q}} \}$ such that 
${\rm{Supp}}\ \eta_{k_{1}...k_{q}} \Subset K_{k_{1}...k_{q}}
\Subset B_{k_{1}...k_{q}}$. 
For the family $K=\{K_{k_{1}...k_{q}} \}$, 
we may assume that the $q$-cochain $\gamma_{\e}$ satisfies  
inequality (\ref{hol-est}). 
By Lemma \ref{rho}, there exists 
a positive constant $C>0$  
such that 
\begin{align*}
| v_{\e} |_{h, \ome} \leq \sum_{k_{1},...,k_{q} \in I}
|\eta_{k_{1}...k_{q}} \wedge \gamma_{\e,  k_{1}...k_{q} } |_{h, \ome} 
\leq C \sum_{ k_{1},...,k_{q} \in I}  \chi_{K_{k_{1}...k_{q}}} |\gamma_{\e,  k_{1}...k_{q} } |_{h, \ome}, 
\end{align*} 
where $\chi_{K_{k_{1}...k_{q}}}$ is the 
characteristic function of $K_{k_{1}...k_{q}}$. 
Note that $C$ depends on the choice of 
$\{ \rho_{i} \}_{i\in I}$, but does not depend on $\e$. 
Therefore we have 
\begin{align*}
\| v_{\e}   \|_{h, \ome} 
\leq C \sum_{ k_{1},...,k_{q} \in I }p_{K_{k_{1}...k_{q}}}(\gamma_{\e})
\end{align*}
from the fundamental inequality 
$(\sum_{i=1}^{N}|a_{i}|)^{2} \leq 2^{N-1} 
\sum_{i=1}^{N}|a_{i}|^{2}$. 
The right hand side can be estimated by a constant independent of $\e$ 
by inequality (\ref{hol-est}).  
This completes the proof. 
\end{proof}

The proof of the following claim  
is based on an argument similar to that of Claim \ref{sol-2}. 
To avoid confusion, 
we use the following notation in the proof. 
\begin{defi}
Let $a_{\e}$ and $b_{\e}$ be $F$-valued $(n,k)$-forms on $Y$. 
We write $a_{\e}\equiv b_{\e}$, 
if there exists an $F$-valued $(n,k-1)$-form $c_{\e}$ on $Y$ 
such that $\dbar c_{\e}= a_{\e} - b_{\e}$ and 
the norm $\| c_{\e} \|_{h_{\e}, \ome}$ is uniformly bounded. 

\end{defi}

\begin{claim}\label{sol-3}
There exists an $F$-valued $(n, q-1)$-form  
$\widetilde{v}_{\e}$ on $Y$  satisfying the 
following properties\,$:$ 
\begin{equation*}
{\rm{(1)}}\hspace{0.2cm} \dbar \widetilde{v}_{\e}+U_{\e}=
g_{\e}(f_{\e}(U_{\e})).   
\quad 
{\rm{(2)}}\hspace{0.2cm} \text{The norm }
\|\widetilde{v}_{\e} \|_{h_{\e}, \ome} 
\text{ is uniformly bounded}. 
\end{equation*}
\end{claim}
\begin{proof}
For $\beta_{\e, i_{0}\dots i_{q-1}}$ with equality ($*$), 
we have $f_{\e}(U_{\e})=\delta \{ \beta_{\e, i_{0}\dots i_{q-1}}\}$. 
We observe  
\begin{align*}\label{eq-1}
\beta_{\e, k_{2}...k_{q}i_{0}} +
\sum_{\ell=2}^{q} (-1)^{\ell-1} \beta_{\e, k_{1}...\hat{k_{\ell}}...k_{q}i_{0}}
+(-1)^{q} \beta_{\e, k_{1}...k_{q}}. 
\end{align*}
\vspace{0.2cm} \\
{\bf{[Argument 4]}}
\vspace{0.1cm}\\
First we consider the second term. 
For an integer $\ell$ with $2 \leq \ell \leq  q$,  
the second term $ \beta_{\e, k_{1}...\hat{k_{\ell}}...k_{q}i_{0}}$ 
is independent of $k_{\ell}$. 
By the same reason as Argument 2 in Claim \ref{sol-2}, 
we can conclude that this term does not affect $g_{\e}(f_{\e}(U_{\e}))$. 
\vspace{0.2cm} \\
{\bf{[Argument 5]}}
\vspace{0.1cm}\\
Secondly we consider the third term. 
Our aim of Argument 5 is to show 
\begin{align*}
\sum_{k_{q}\in I} \dbar \rho_{k_{q}} \wedge  
\sum_{k_{q-1}\in I} \dbar  \rho_{k_{q-1}} \wedge  \cdots \wedge 
\sum_{k_{2}\in I} \dbar \rho_{k_{2}}  \wedge 
\sum_{k_{1}\in I} \dbar(\rho_{k_{1}} \beta_{\e, k_{1}...k_{q}})
\equiv 0. 
\end{align*}
When we define $\eta_{k_{2}...k_{k_{q}}}$ by 
\begin{equation*}
\eta_{k_{2}...k_{k_{q}}}:=
\rho_{k_{q}}
\dbar  \rho_{k_{q-1}} \wedge
\dbar \rho_{k_{q-2}}
\cdots  \wedge \dbar \rho_{k_{2}},   
\end{equation*}
the left hand side agrees with 
\begin{align*}
\dbar \Big( \sum_{k_{1}, \dots, k_{q}\in I}  \eta_{k_{2}...k_{k_{q}}}  \wedge 
\big( \dbar \rho_{k_{1}} \wedge \beta_{\e, k_{1}...k_{q}} +
\rho_{k_{1}} \wedge \dbar \beta_{\e, k_{1}...k_{q}} \big) \Big)
\end{align*}
by the Leibnitz rule. 
By Lemma \ref{rho} and the fundamental inequality $|a+b|^2 \leq 2(|a|^2+|b|^2)$, 
we obtain 
\begin{align*}
&\Big| \sum_{k_{1}, \dots, k_{q}\in I}  \eta_{k_{2}...k_{k_{q}}}  \wedge 
\big( \dbar \rho_{k_{1}} \wedge \beta_{\e, k_{1}...k_{q}} +
\rho_{k_{1}} \wedge \dbar \beta_{\e, k_{1}...k_{q}} \big)
\Big|^2_{h_{\e}, \ome} \\
\leq &
C \big( \big| 
\beta_{\e, k_{1}...k_{q}}\big|^2_{h_{\e}, \ome}+
\big| \dbar \beta_{\e, k_{1}...k_{q}} 
\big|^2_{h_{\e}, \ome} \big)
\end{align*}
for some positive constant $C>0$. 
The norms 
$\| \beta_{\e, k_{1}...k_{q}}\|_{h_{\e}, \ome}$ and $
\| \dbar \beta_{\e, k_{1}...k_{q}} 
\|_{h_{\e}, \ome}$ 
can be  
estimated by a constant independent of $\e$  
by the construction (see the proof of Proposition \ref{DWiso}).    
Therefore we have  
\begin{align*}
\sum_{k_{q}\in I} \dbar \rho_{k_{q}} \wedge  
\sum_{k_{q-1}\in I} \dbar  \rho_{k_{q-1}} \wedge  \cdots \wedge 
\sum_{k_{2}\in I} \dbar \rho_{k_{2}}  \wedge 
\sum_{k_{1}\in I} \dbar(\rho_{k_{1}} \beta_{\e, k_{1}...k_{q}})
\equiv 0. 
\end{align*}
\ 
\\
{\bf{[Argument 6]}}\\
Finally we consider the first term. 
Our aim is to show 
\begin{equation*}
\sum_{k_{q}\in I} \dbar \rho_{k_{q}} \wedge  
\sum_{k_{q-1}\in I} \dbar  \rho_{k_{q-1}} \wedge  \cdots \wedge 
\sum_{k_{2}\in I} \dbar \rho_{k_{2}}  \wedge 
\sum_{k_{1}\in I} \dbar(\rho_{k_{1}} \beta_{\e, k_{2}...k_{q}i_{0}})
\equiv U_{\e}. 
\end{equation*}
Since $\beta_{\e, k_{2}...k_{q}i_{0}}$ does not depend on $k_{1}$, 
we have 
\begin{align}\label{eq-2}
\sum_{k_{1}\in I} \dbar(\rho_{k_{1}} \beta_{\e, k_{2}...k_{q}i_{0}})
&= 
\dbar  \beta_{\e, k_{2}...k_{q}i_{0}}\\ \notag
&=\beta_{\e, k_{3}...k_{q}i_{0}} + 
\sum_{\ell=3}^{q} (-1)^{\ell}
\beta_{\e, k_{2}...\hat{k_{\ell}}...k_{q}i_{0}}+ 
(-1)^{q+1} \beta_{\e, k_{2}...k_{q}}.  
\end{align}
Note that the second equality follows from equality ($*$). 
The second term of the right hand side of (\ref{eq-2}) 
does not affect, by the same reason as 
Argument 2 (Argument 4). 
Moreover, for the third term of the right hand side, 
we can show  
\begin{align*}
\sum_{k_{q}\in I} \dbar \rho_{k_{q}} \wedge 
\sum_{k_{q-1}\in I} \dbar  \rho_{k_{q-1}}  \wedge \cdots \wedge
\sum_{k_{2}\in I} \dbar  (\rho_{k_{2}} \beta_{\e, k_{2}...k_{q}})
\equiv 0
\end{align*}
by the same method as Argument 5. 
In summary, 
we have proved   
\begin{align*}
&\sum_{k_{q}\in I} \dbar \rho_{k_{q}} \wedge  
\sum_{k_{q-1}\in I} \dbar  \rho_{k_{q-1}} \wedge  \cdots \wedge 
\sum_{k_{2}\in I} \dbar \rho_{k_{2}}  \wedge 
\sum_{k_{1}\in I} \dbar(\rho_{k_{1}} \beta_{\e, k_{2}...k_{q}i_{0}})\\ 
\equiv &
\sum_{k_{q}\in I} \dbar \rho_{k_{q}} \wedge  
\sum_{k_{q-1}\in I} \dbar  \rho_{k_{q-1}} \wedge  \cdots \wedge 
\sum_{k_{2}\in I} \dbar \rho_{k_{3}}  \wedge 
\sum_{k_{1}\in I} \dbar(\rho_{k_{2}} \beta_{\e, k_{3}...k_{q}i_{0}}).  
\end{align*}
By repeating this procedure, 
we obtain
\begin{align*}
&\sum_{k_{q}\in I} \dbar \rho_{k_{q}} \wedge  
\sum_{k_{q-1}\in I} \dbar  \rho_{k_{q-1}} \wedge  \cdots \wedge 
\sum_{k_{2}\in I} \dbar \rho_{k_{2}}  \wedge 
\sum_{k_{1}\in I} \dbar(\rho_{k_{1}} \beta_{\e, k_{2}...k_{q}i_{0}})\\ 
\equiv &  
\sum_{k_{q}\in I} \dbar  \rho_{k_{q}} \wedge
\sum_{k_{q-1}\in I}\dbar (\rho_{k_{q-1}}\beta_{\e, k_{q}i_{0}}) \\
= &  \sum_{k_{q}\in I}\dbar \rho_{k_{q}} \wedge
\dbar \beta_{\e, k_{q}i_{0}}\\ 
= &  \sum_{k_{q}\in I} \dbar \big( \rho_{k_{q}} 
(\beta_{\e, i_{0}} - \beta_{\e, k_{q}})\big).  
\end{align*}
The last equality follows from equality ($*$). 
The norm of $\beta_{\e, k_{q}}$ can be estimated by a constant of independent of $\e$ and 
$\dbar \beta_{\e, i_{0}} =  U_{\e}$ holds on $B_{i_{0}} \setminus Z$ by the construction. 
Hence we have 
\begin{align*}
\dbar \sum_{k_{q}\in I} \big( \rho_{k_{q}}\beta_{\e, k_{q}} \big)
 \equiv 0 \text{\quad and \quad}
\dbar \big( \sum_{k_{q}\in I}\rho_{k_{q}} \beta_{\e, i_{0}} \big)= 
\dbar  \beta_{\e, i_{0}} =  U_{\e}. 
\end{align*}
This completes the proof.
\end{proof}

From Claim \ref{sol-2} and \ref{sol-3} we can obtain the 
conclusion. 
Indeed, if we put $V_{\e}:=
v_{\e} - \widetilde{v}_{\e}$, then $V_{\e}$ 
satisfies the properties in Theorem \ref{cru}. 
\end{proof}


\end{document}